\documentclass[11pt,oneside, reqno]{amsart} % use larger type; default would be 10pt
\usepackage[utf8x]{inputenc} % set input encoding (not needed with XeLaTeX)

\usepackage{geometry} % to change the page dimensions
\usepackage{graphicx} % support the \includegraphics command and options
\usepackage{color}
\usepackage{stmaryrd}

\usepackage{booktabs} % for much better looking tables
\usepackage{array} % for better arrays (eg matrices) in maths
\usepackage{paralist} % very flexible & customisable lists (eg. enumerate/itemize, etc.)
\usepackage{verbatim} % adds environment for commenting out blocks of text & for better verbatim
\usepackage{subfig} % make it possible to include more than one captioned figure/table in a single float
\usepackage{float}

\usepackage{amsfonts}
\usepackage{dsfont}
\usepackage{amsmath}
\usepackage{amsthm}
\usepackage{amssymb}
\usepackage{mathtools}
\usepackage{pgfplots,pgffor}
\mathtoolsset{centercolon,showmanualtags}

\newcommand{\mel}{\MoveEqLeft}
\newcommand{\la}{\langle}
\newcommand{\ra}{\rangle}
\usepackage{amsthm}
\theoremstyle{definition}

\newtheorem{theorem}{Theorem}[section]
\newtheorem{definition}[theorem]{Definition}
\newtheorem{remark}[theorem]{Remark} 
 
\newtheorem{example*}[theorem]{Example\textsuperscript{*}}
\newtheorem{proposition*}[theorem]{Proposition\textsuperscript{*}}
\newtheorem{corollary}[theorem]{Corollary} 
\newtheorem{corollary*}[theorem]{Corollary\textsuperscript{*}}
\newtheorem{proposition}[theorem]{Proposition} 
\newtheorem{lemma}[theorem]{Lemma}

\newcommand{\psiin}{\psi^{\text{in}}}
\newcommand{\psiout}{\psi^{\text{out}}}
\newcommand{\phin}{\varphi^{\text{in}}}
\newcommand{\phout}{\varphi^{\text{out}}}
\newcommand{\rhoin}{\rho^{\text{in}}}
\newcommand{\rhoout}{\rho^{\text{out}}}

\newcommand{\In}{\text{in}}
\newcommand{\Out}{\text{out}}

\def\Limes#1#2 {\lim\limits_{#1\rightarrow #2}}

\DeclareMathOperator{\Id}{Id}

\newcommand{\We}{\operatorname{We}}

\def\eps{\varepsilon}

\def\R{\mathbb{R}}
\def\B{\mathcal{B}}

\newcommand{\Sp}{\mathbb{S}}

\def\sym{\mathrm{sym}}
\def\N{\mathbb{N}}

\def\XXint#1#2#3{{\setbox0=\hbox{$#1{#2#3}{\int}$ }
\vcenter{\hbox{$#2#3$ }}\kern-.59\wd0}}

\newcommand{\abs}[1]{\left\lvert#1\right\rvert}
\newcommand{\1}{\mathds{1}}

\def\const{\mathrm{const}}

\def\curl{\mathrm{curl} \,}
\newcommand{\K}{\mathcal{K}}
\newcommand{\A}{\mathcal{A}}
\renewcommand{\S}{\mathcal{S}}
\newcommand{\G}{\mathcal{G}}

\newcommand{\cH}{\mathcal{H}}

\newcommand{\J}{\mathcal{J}}

\newcommand{\cC}{\mathcal{C}}
\newcommand{\M}{\mathcal{M}}
\newcommand{\D}{\mathrm{D}}
\newcommand{\grad}{\nabla}

\DeclareMathOperator{\Span}{Span}

\renewcommand{\L}{\mathcal{L}}
\def\norm#1{\lVert #1 \rVert}
\def\scalar#1#2{\langle #1,#2 \rangle}

\renewcommand{\div}{\operatorname{div}}

\newcommand{\dx}{\, \mathrm{d}}

\newcommand{\laplace}{\Delta}
\renewcommand{\div}{\grad\cdot}
\renewcommand{\epsilon}{\varepsilon}

\newcommand{\psiH}{\psi_{\mathrm{S}}}
\newcommand{\VH}{V_{\mathrm{S}}}

\newcommand{\F}{\mathcal{F}}

\newcommand{\din}{\delta^{\text{in}}}
\newcommand{\dout}{\delta^{\text{out}}}

\newcommand{\wein}{\gamma}
\newcommand{\weout}{\We}

\newcommand{\Bin}{\mathcal{D}^{\text{in}}}
\newcommand{\Bout}{\mathcal{D}^{\text{out}}}

\newcommand{\jump}[1]{{\left\llbracket{#1}\right\rrbracket}}
\usepackage[square]{natbib}
\setcitestyle{numbers}

\usepackage{hyperref}
\numberwithin{equation}{section}
\usepackage{xcolor}

\renewcommand{\L}{\operatorname{L}} %Lebesgue spaces
\newcommand{\h}{\operatorname{h}}
\newcommand{\C}{\operatorname{C}} %spaces of continuous functions
\renewcommand{\H}{\operatorname{H}} %Sobolev spaces

\begin{document}

\title{Steady bubbles and drops in inviscid fluids}

\thanks{LN and CS are funded by the Deutsche Forschungsgemeinschaft (DFG, German Research Foundation) under Germany's Excellence Strategy EXC 2044--390685587, Mathematics M\"unster: Dynamics Geometry Structure, and grant  531098047. DM has received funding from the European Research Council (ERC) under the European Union’s Horizon 2020 research and innovation programme through the grant agreement 862342.
}
\date{\today}

\author{David Meyer}
\email[David Meyer]{david.meyer@icmat.es}

\author{Lukas Niebel}
\email[Lukas Niebel]{lukas.niebel@uni-muenster.de}

\author{Christian Seis}
\email[Christian Seis]{seis@uni-muenster.de}

\address[David Meyer]{Instituto de Ciencias Matemáticas, Calle Nicol\'as Cabrera 13-15, 28049 Madrid, Spain}

\address[Lukas Niebel and Christian Seis]{Institut f\"ur Analysis und Numerik, Universit\"at M\"unster\\
Orl\'eans-Ring 10, 48149 M\"unster, Germany.}

\begin{abstract}
We construct steady non-spherical bubbles and drops, which are traveling wave solutions to the axisymmetric two-phase Euler equations with surface tension, whose inner phase is a bounded connected domain. The solutions have a uniform vorticity distribution in this inner phase and they have a vortex sheet on its surface. 

Our construction relies on a perturbative approach around an explicit spherical solution, given by Hill's vortex enclosed by a spherical vortex sheet.
The construction is sensitive to the Weber numbers describing the flow. At critical Weber numbers, we perform a bifurcation analysis utilizing the Crandall--Rabinowitz theorem in Sobolev spaces on the $2$-sphere. Away from these critical numbers, our construction relies on the implicit function theorem.

Our results imply that the model containing  surface tension is richer than the ordinary one-phase Euler equations, in the sense that for the latter,  Hill's spherical vortex  is unique (modulo translations) among all axisymmetric simply connected uniform vortices of a given circulation.
\end{abstract}
\maketitle  
\allowdisplaybreaks

\section{Introduction}
\subsection{Motivation}

Studying the motion and shapes of bubbles and drops in liquids is a fundamental problem in fluid dynamics. It has received considerable attention over more than half a century both by theoretical physicists and experimentalists. The goal of the present work is to lay some of their studies on sound mathematical ground by rigorously constructing equilibrium bubble and drop solutions to the underlying mathematical model equations.

Asymptotic analyses of nearly spherical bubbles and drops in the physics literature have proceeded largely in parallel. By a bubble, we mean a nearly spherical volume of gas immersed in a liquid. By a drop in a liquid, we refer to a nearly spherical volume of one liquid inside another immiscible liquid of comparable density (and viscosity). We will focus on high Reynolds number flows, which will allow us to consider the two-phase Euler equations with surface tension as the underlying model equations.

When a drop or a bubble starts moving by the influence of gravity, the resting ambient liquid induces vorticity on its boundary. The vorticity is subsequently diffused both inwards and outwards of the drop or bubble, and it homogenizes at high Reynolds numbers \cite{Prandtl1905,Batchelor1956}:
In the interior, it generates an inner circulation of uniform vorticity distribution, which is negligible in the case of a bubble, while in an infinite ambient fluid, the flow becomes potential. Thin viscous boundary layers that typically emerge along the surfaces \cite{Moore1963-fm,harper1968motion} formally disappear in the high-Reynolds number limit. Instead, a vortex sheet forms due to a discontinuity in the tangential velocity components, with its motion driven entirely by circulation rather than gravitational effects. We analyze both scenarios where the circulation inside the vortex body is aligned with and opposite to the circulation on the vortex sheet at the boundary.

By choosing an inviscid model, we also discard any viscous wakes inside and past the moving objects.  Because we gain time reversibility, we may even assume symmetry with respect to the co-moving reference plane. Considering inviscid models for moving bubbles and drops in liquids at high Reynolds numbers was justified for instance in \cite{Moore1963-fm,harper1968motion,Pozrikidis_1989}. The behavior of bubbles and drops in liquids of intermediate or small Reynolds numbers can be significantly different from that in high Reynolds number flows, see for instance \cite{loth2008quasi}.

The two-phase Euler equations describing the translating motion of a bubble or a drop allow for exactly two dimensionless parameters. These are the commonly used \emph{Weber number} $\We$ that measures the inertial forces relative to surface tension forces, and a second quantity $\wein$ measuring the vortex strength relative to surface tension forces, see \eqref{40} below. The latter is sometimes referred to as the \emph{vortex Weber number}, see, e.g., \cite{Hendrickson_Yu_Yue_2022}.

In the situation in which the two dimensionless parameters agree, $\weout=\wein$, the two-phase Euler equations allow for a spherical traveling wave solution: As the vorticity is uniform, the interior flow inside a sphere must be identical to that inside Hill's spherical vortex \cite{Hill1894-ia}. The exterior flow is potential, and there will be a jump discontinuity of the tangential velocity across the surface if the mass densities of the inner and outer phases are different. This traveling wave solution is thus a spherical vortex sheet enclosing Hill's vortex. It will be a central object in our analysis. In order to simplify our language use, we will refer to this object as the \emph{spherical vortex}. 

In the present paper, we rigorously construct slightly non-spherical bubbles and drops in situations when the two Weber numbers are close to each other, $\weout\approx\wein$. On the one hand, we find critical values $(\gamma_k)_{k\in\N}$ at which such objects with identical parameters $\We=\gamma$ bifurcate along a curve from the spherical vortex at $\gamma_k$. On the other hand, away from these critical Weber numbers, we find for any $\gamma$ a small open neighborhood of steady non-spherical solutions with $\We\not=\gamma$. Moreover, we will establish that our non-spherical vortex configurations converge towards a spherical vortex in the limit of vanishing Weber numbers, which is realized in the limit of large surface tensions. The latter applies also to situations, in which either the inner or the outer phase is a vacuum. If the inner phase is a vacuum, our configurations are \emph{hollow vortices}. Our asymptotic results for small Weber numbers confirm earlier predictions in the physics literature on the spheroidal shapes, for instance,  \cite{Moore1959-gv,Harper1972-lf,Pozrikidis_1989}, which are based on formal asymptotics.

\subsection{Mathematical model}
The two-phase Euler equations with surface tension are given by the following free boundary problem:
\begin{alignat}{2} \label{eq:2pheuler}
\rho\left(\partial_t U+U \cdot \nabla U\right)+\nabla P &=0 && \quad  \text { in } \mathbb{R}^3 \backslash \mathcal{S}(t) ,\\
\nabla \cdot U & =0 &&\quad\text { in } \mathbb{R}^3, \label{10}\\
\jump{P } & ={\sigma} H &&\quad \text { on } \mathcal{S}(t),\label{11} \\
\jump{U \cdot n } & =0 &&\quad \text { on } \mathcal{S}(t),\label{12} \\
\nu & = U \cdot n  &&\quad\text { on } \mathcal{S}(t).\label{13}
\end{alignat}
Here, $U$ is the velocity of the fluid, $P$ is the hydrodynamic pressure, and $\rho$ is the mass density of the respective phases. The first two equations model thus the conservation of momentum \eqref{eq:2pheuler} away from the interface $\S(t)$, and the incompressibility  \eqref{10} of the fluid. The Young--Laplace equation \eqref{11}, in which $\sigma >0$ is the surface tension and $H$ is (twice) the mean curvature, relates the difference in pressure to the geometry of the interface. The normal component of the velocity is continuous across the interface \eqref{12}, whose normal velocity $\nu$ is that of the fluid \eqref{13}.
Motivated by bubbles and drops, we will in the following assume that one phase is bounded and connected, occupying an inner domain $\Bin(t)$. The ambient fluid will be denoted by $\Bout(t) = \R^3\setminus \overline{\Bin(t)}$, so that $\mathcal{S}(t)=\partial \Bin(t)=\partial \Bout(t)$. The two phases are characterized by their respective densities,
$$
\rho(t)=\rhoin \1_{\Bin(t)}+\rhoout \1_{\Bout(t)}
$$
with $ \rhoin,\rhoout \ge0 $. We choose the mean curvature $H$  to be positive for convex $\Bin$, and normalize it to $H=2$ if $\Bin$ is the unit ball. The brackets $\jump{f}$ measure the jump of the quantity $f$ across the interface, $\jump{f}=f_{\Bin}-f_{\Bout}$, where $f_{\Bin}$ denotes the limit from the inner phase and $f_{\Bout}$ the limit from the outer phase.  We let the unit normal vector $n$ point from the inner phase into the outer phase, and $s$ is then the velocity of the surface in the direction of that normal. 
While the normal velocity component is necessarily continuous across the interface \eqref{12}, its tangential components will, in general, experience jump discontinuities, turning the surface $\S(t)$ into a vortex sheet. Indeed, a short exercise reveals that the distributional vorticity has a singular part given by $\jump{n\times U}$, concentrated on the surface.

Our goal in this paper is the construction of traveling wave solutions with $\S(t)$ close to a steadily translating sphere of some fixed speed $V \ge 0$ and with vorticity distributed both in the interior $\Bin(t)$ and on its boundary $\S(t)$. Before specifying our mathematical setting, we provide a short overview of the related mathematical literature.

We are concerned with a free boundary problem for the Euler equations, where $\S(t)$ is the free boundary. In the absence of surface tension, it is well known that the interface problem between two inviscid and incompressible fluids is ill-posed due to the Kelvin--Helmholtz instability \cite{Ebin1988-qc,Wu2006-zw}. The two-phase Euler equations with surface tension attracted considerable attention in the past years, primarily in the context of studying water waves. In the irrotational case, that is, if the curl of the velocity vector is zero within the two phases and the vorticity is thus concentrated on the interface, the model can be completely reduced to the interface evolution. Local well-posedness results for the two-phase Euler equations with surface tension were obtained in \cite{Iguchi1997-vu,Ambrose2003-wy,schweizer2005three,coutand2007well,Ambrose2007-qd,Cheng2008-nu}. In \cite{Shatah2008-gd} short-time regularity of the velocity and the interface is proven. Singularity formation in finite time is shown e.g.\ in \cite{castro2012finite,coutand2014finite}.

Regarding the mathematically rigorous phenomenological study of three-dimensional vortex structures, there is a considerable number of results in the one-fluid setting. An important example is Hill's explicit solution for a spherical vortex \cite{Hill1894-ia} from 1894, which we recall in detail and extend to the two-fluid setting below.  Given a fixed radius and circulation, it was proven in \cite{Amick1986-al} to be unique modulo translations. We refer to \cite{Choi2024-mx}, where stability with respect to axisymmetric {(with respect to the $x_3$-axis)} perturbations is derived. The introduction of that paper also gives an extensive overview of the literature on Hill's spherical vortex. Vortex \emph{rings}, which are toroidal vortex configurations that are also observed to emerge out of rising air bubbles as a result of gravity \cite{Lundgren1991-hl,Bonometti2006-zt}, were mathematically constructed in \cite{Norbury1972-cz,Ambrosetti1989-qh}, and the leapfrogging interaction or vortex rings was recently investigated in   \cite{butta2025leapfrogging,Davila2022-ke}. More general vortex filaments are studied in \cite{Jerrard2017-ve}. Traveling ring-shaped vortex sheets in the one-fluid setting and toroidal vortex bubbles in the two-phase setting were recently constructed in  \cite{Meyer2024-xx}.   In the low Reynolds number setting, equilibrium configurations of falling drops were constructed in \cite{bemelmans1981liquid,solonnikov1999problem,eiter2023falling}. 
Travelling wave solutions were also constructed in the context of water waves with surface tensions  \cite{Walsh2014-cf,Kinnersley1976-qz,Crapper1957-qp,Moon2024-on,Stevenson2023-of,shatah2013travelling,Wahlen2006-qf,Disconzi2016-xb}. The phenomenon in which steady fluid equilibria bifurcate as surface tension varies occurs in many other contexts; see, for instance, \cite{Crapper1957-qp,Wegmann2000-ql,Moon2024-on,Baldi2024-bi,Wahlen2006-qf,murgante2024steadyvortexsheetspresence}.

\subsection{Traveling wave solutions}
 
In order to provide a first formulation of our results, we introduce our notion of traveling wave solutions.  Upon a rotation of the coordinates system, we may assume that our bubbles or drops move at speed $V{\geq 0}$  in the direction of $x_3$, assuming that the speed is non-negative is non-restrictive because of the time-reversibility of the problem. Therefore, it is convenient to write the problem in the moving frame as
\begin{equation*}
	u(x)=U(t, x_1,x_2, x_3+V t)-V e_{3}, \; p(x)=P(t, x_1,x_2, x_3+V t), \; \mathcal{S}(t)=\mathcal{S}+t V e_3,
\end{equation*}
where $x = (x_1,x_2,x_3) \in \R^3$ and $t>0$. From now on we will only consider the time-independent quantities $u,p,\S$, which solve the steady two-phase Euler equations 
\begin{alignat}{2}
\rho \, (u \cdot \nabla) u +\nabla p & =0 &&\quad \text { in } \mathbb{R}^3 \backslash \mathcal{S},\label{2a} \\
\nabla \cdot u & =0 &&\quad \text { in } \mathbb{R}^3,\label{2b} \\
\jump{p } & ={\sigma} H &&\quad \text { on } \S, \label{eq:jump} \\
 u \cdot n   & =0 &&\quad \text { on } \S. \label{2d}
\end{alignat}
It is important to notice that in the moving frame, $u$ does not vanish at infinity anymore, instead (assuming the solution is sufficiently regular) we have 
\begin{align*}
\lim_{|x|\rightarrow \infty} u(x)=-Ve_3.
\end{align*}
It is well-known that in the steady setting \eqref{2a}, the pressure can be determined via the Bernoulli equations,
\begin{equation*}
	\frac{\rho^{\mathrm{in}}}{2}\left|u^{\mathrm{in}}\right|^2+p^{\mathrm{in}}=\const, \quad \frac{\rho^{\mathrm{out}}}{2}\left|u^{\mathrm{out}}\right|^2+p^{\mathrm{out}}=\const.
\end{equation*}
We may thus rewrite the interfacial condition \eqref{eq:jump} as
\begin{equation}
	\frac{1}{2}\jump{\rho|u|^2}+{\sigma} H=\const \; \mbox{ on } \S.\label{bjump}
\end{equation}
The constant on the right-hand side is an unknown of the problem. It will not be of importance for our analysis, and thus, it will not get a name.

We aim to construct traveling wave solutions in the class of axisymmetric and swirl-free velocity fields. As can be easily checked, the curl of such vector fields points into the azimuthal direction. 
We restrict our attention to configurations that feature a uniform vorticity distribution in the interior of the drop or bubble, which means that the vorticity  vector $\omega_a=\curl u$ in the interior is given by 
\begin{equation}\label{vort inside}
	\omega_a = {\frac{15}2}a \begin{pmatrix}
		-x_2 \\ x_1 \\0 
	\end{pmatrix} \text{in }\Bin,
\end{equation}
for some vorticity $a \in \R$. The prefactor $15/2$ is introduced here to have a simpler notation later on. We note that $\Bin$ must be invariant under rotations in the azimuthal direction,
\[
x\in \Bin \quad \Longleftrightarrow\quad R(\varphi)x\in\Bin\quad\mbox{for all }\varphi \in [0,2\pi),
\]
where $R(\varphi)\in\R^{3\times 3}$ is the matrix generating rotations of angle $\varphi$ in the $(x_1,x_2)$ plane. We remark that \eqref{vort inside} states that the azimuthal \emph{potential} vorticity is constant, which is a reasonable assumption, because its value is an integral of the Euler equations, see \cite[Section 2.3.3]{Majda2001-ed}. In the ambient fluid, on the other hand, we assume that the flow is irrotational, that is 
\begin{equation}
\curl u=0\quad\text{ in $\Bout$}.\label{liq vort}
\end{equation}
It will be convenient to introduce the vector stream function rather than the axisymmetric scalar stream function to represent the velocity field, see, for instance, \cite[Section 2.4]{Majda2001-ed}. We write
\begin{equation}\label{16}
	u = \curl \psi-Ve_3,
\end{equation}
where $\psi \colon \R^3 \to \R^3$ is a divergence-free  vector field. Because $u$ is axisymmetric without swirl, the identity simplifies to 
\[
u =  - \partial_z \psi_\varphi e_r + \left(\frac1r\psi_{\varphi} +\partial_r\psi_{\varphi}-V\right) e_z,\quad \mbox{and}\quad \partial_z \psi_r = \partial_r\psi_z
\]
where $\psi = \psi_r e_r+\psi_z e_z + \psi_\varphi e_\varphi$ in cylindrical coordinates $(r,z,\varphi) \in [0,\infty) \times \R \times [0,2\pi)$. Using the divergence-free condition for $\psi$, which reads $\partial_r \psi_r + r^{-1}\psi_r +\partial_z \psi_z=0$ in cylindrical coordinates, implies that we can choose the vector stream function purely azimuthal, $\psi_r=\psi_z=0$. The azimuthal function $r\psi_{\varphi}$ coincides thus with the axisymmetric stream function. 

The tangential flow condition \eqref{2d} on the velocity finally yields that the interface $\S$ is a stream surface, $r\psi_{\varphi}- r^2 V/2 = \const$. There is thus no loss of generality to suppose that
\begin{equation}
    \label{15}
    \psi = \frac{V}2re_{\varphi}\quad\mbox{on }\S.
\end{equation}
 Using the identity $\curl\curl = \nabla\div -\Delta$, we eventually notice that the vector stream function solves the elliptic problem 
\begin{equation} \label{eq:vorticity}
	-\Delta \psi = \omega_a \mathds{1}_{\Bin}\quad\mbox{in $\R^3 \setminus \S$.}
\end{equation} 
Moreover, invoking the representation \eqref{16}, we further rewrite the jump condition \eqref{bjump} as
\begin{equation}\label{eq:jumpcond}
	\frac{1}{2}\jump{\rho|\curl \psi-Ve_3|^2}+{\sigma} H=\const \quad \text { on } \S.
\end{equation}

\noindent We remark that \eqref{15}, \eqref{eq:vorticity}, and \eqref{eq:jumpcond} constitute an \emph{overdetermined} boundary problem. Indeed, the elliptic problem \eqref{eq:vorticity} is posed on the interior and exterior domain equipped with both Dirichlet boundary conditions \eqref{15} and the jump condition \eqref{eq:jumpcond} for the curl of $\psi$. For arbitrary domains $\Bin$, both conditions may not be simultaneously verified. Requiring both conditions simultaneously thus imposes a restriction on the shape of the domain. Our goal in the present paper is to find a surface $\S$, so that the combined problem \eqref{15},  \eqref{eq:vorticity}, and \eqref{eq:jumpcond} admits a solution. Once this solution is found, the velocity field $u$ defined in \eqref{16} is a solution to the Euler system \eqref{2a}, \eqref{2b}, \eqref{eq:jump}, and \eqref{2d}.

As announced in the first subsection, one explicit solution can be found for spherical surfaces $\S =\Sp^2_R$, where $\Sp^2_R$ denotes the $2$-sphere of radius $R$. Here, the interior solution describes a spherical vortex derived already by Hill in 1894 \cite{Hill1894-ia}, and the outer flow is potential. The \emph{spherical solution} is given by
\begin{equation} \label{eq:psiHill} 
	\psiH(x) = \begin{pmatrix}
		-x_2 \\ x_1 \\ 0
	\end{pmatrix} \cdot \begin{cases}
		{\frac{3a}4}\left(R^2-\abs{x}^2\right)+\frac{\VH}{2}&\mbox{for }\abs{x}\le R \\[2mm]
		\frac{\VH}{2} \frac{R^3}{\abs{x}^{3}}&\mbox{for }\abs{x}> R,
		    \end{cases}
\end{equation}
where, motivated by \eqref{15}, $\VH$ denotes the speed of the vortex configuration. It can be easily checked that this function solves the Dirichlet problem \eqref{15}, \eqref{eq:vorticity} for any choice of velocity. The precise value of $\VH$ is determined by the jump condition \eqref{eq:jumpcond}, which in spherical coordinates reads 
\begin{equation}
	\frac{1}{2}\jump{\rho|\curl \psiH -\VH e_3|^2} =\frac{9}{8}\left(  a^2 R^4 \rhoin- \rhoout \VH^2\right)  \sin^2 \theta.
    \label{exp jump}
\end{equation}
We explain our notational conventions at the end of this first section, see page \pageref{notation}. The function in \eqref{eq:psiHill} defines a solution to \eqref{15}, \eqref{eq:vorticity}, and \eqref{eq:jumpcond} precisely if
\begin{equation}\label{14}
    \VH = \abs{a} R^2\sqrt{\frac{\rhoin}{\rhoout}}.
\end{equation}
The occurrence of the modulus in this relation is interesting: It shows that there  exist spherical solutions traveling in the direction of $x_3$  (so that $\VH>0$)   both with positive and negative inner circulations $a$. A negative inner circulation has then to be compensated by a counter-rotating outside flow, so that the overall circulation is again positive. In fact, a short computation of $U = \curl \psiH$ shows that the function in \eqref{eq:psiHill} generates a vortex sheet on the surface if $\VH\not=a R^2$. In  particular, even in the one-fluid setting, $\rhoin=\rhoout$, a spherical vortex sheet solution exists if \eqref{14} is satisfied for a negative $a$. For $a$ positive and equal densities, the velocity field is continuous and $\psiH$ is precisely the vector stream function that corresponds to Hill's spherical vortex in the one-fluid setting.

In Section \ref{sec3}, we will non-dimensionalize the problem \eqref{15}, \eqref{eq:vorticity}, and \eqref{eq:jumpcond}. As mentioned in the introduction, we will see that all system constants can be reduced to two dimensionless control parameters: the Weber number $\We$ and the vortex Weber number $\gamma$, given by
\begin{equation}\label{40}
     \weout =\frac{\rhoout V^2R }{\sigma}  , \quad \wein = \frac{\rhoin a^2R^5 }{\sigma}  .
\end{equation}
For near-spherical vortices, $R$ denotes their equivalent radius. The first quantity measures the inertial forces relative to the surface tension forces, while the second one measures the vortex intensity relative to the surface tension forces.  It is readily checked that   condition \eqref{14} under which the spherical function \eqref{eq:psiHill} solves the jump condition \eqref{eq:jumpcond} can now be equivalently stated as
\begin{equation}\label{19}
    \weout =   \wein.
\end{equation}

\noindent In this paper, we will look for solutions to the overdetermined problem \eqref{15}, \eqref{eq:vorticity}, and \eqref{eq:jumpcond}  which are close to the spherical one in \eqref{eq:psiHill}, \eqref{14}. More precisely, we will consider surfaces that can be written  as a graph over $\Sp^2_R$, i.e., for $\eta\in \C^0(\Sp_R^2)$ we set \begin{equation}
	\mathcal{S}=\S_\eta = \left\{ (1+\eta(x))x : x \in \Sp^2_R \right\}. \label{Sgraph}
\end{equation}
This is the boundary of a simply connected bounded set if $\eta$ is continuous and $\eta>-1$ (which is always going to be the case below). Any $\S_\eta$  for which \eqref{15}, \eqref{eq:vorticity}, and \eqref{eq:jumpcond}  are simultaneously solvable yields a traveling wave solution to the two-phase Euler equations.

We focus on axisymmetric functions that experience a reflection symmetry with respect to the reference plane. For simplicity, we call such functions simply \emph{symmetric}. Then $f$ is symmetric precisely if
\begin{equation}\label{43}
f = f(\theta)\quad \mbox{and}\quad f\left(\frac{\pi}2-\theta\right)=f\left(\frac{\pi}2+\theta\right).
\end{equation}
In this notation, symmetric functions are $\pi$-periodic functions of the polar angle $\theta$.

We give now the first version of our main result.

\begin{theorem} \label{thm:bubbles}
    There exists an increasing sequence $\Gamma = (\gamma_k)_{k\in\N}$ of positive numbers diverging to infinity as $k\to \infty$ with the following property:
    \begin{enumerate}
        \item 
        For any $\gamma\in [0,\infty)\setminus \Gamma$ and any $\weout$ close to but different from $\gamma$, there exists a steady, symmetric, and smooth solution of \eqref{15}-\eqref{eq:jumpcond} with a volume of $4/3\pi R^3$ that is nearly spherical. This solution is the only steady, non-spherical, and symmetric smooth solution of that volume in a small neighborhood of the spherical vortex with parameter $\gamma$. 
        
	\noindent Moreover, if $
    \wein = \eps\din$ and $ \weout = \eps\dout$ 
    for some $\delta^\In,\delta^\Out \in [0,\infty)$ with $
    \din\not=\dout$ and $\eps\ll1$, the radial distance $d_{\eps}$ from any point of the  sphere $\Sp_R^2$ to the surface of the constructed object  satisfies the 
     asymptotic expansion  
     \begin{equation}\label{eq:largasym}
         d_{\eps} = \eps R \frac{3}{32}(\din-\dout)\left(3\cos^2\theta-1\right) +o(\eps),
     \end{equation}
	as $\eps\to0$. 
        \item For any $k\in\N$, there exists a unique curve of steady, non-spherical, symmetric, and smooth solution of \eqref{15}-\eqref{eq:jumpcond} with a volume of $4/3\pi R^3$ and with Weber numbers $\gamma=\weout$ close to and bifurcating from the spherical vortex with parameter $\gamma_k$. 
    \end{enumerate}
Furthermore, we have the explicit lower bound $\gamma_1 > 1.861$.
\end{theorem}

Whether the constructed objects are bubbles or drops depends on the particular choices of the densities $\rhoin$ and $\rhoout$.

Uniqueness even holds true in a large class of Sobolev functions, which we will explain in a more precise version of Theorem  \ref{thm:bubbles}, namely Theorem \ref{T2} in Section \ref{sec3}. We will collect further comments on Theorem \ref{thm:bubbles} in Remark \ref{weber number} below. 

 \begin{figure}[H]
 \begin{minipage}{0.3\linewidth}
 %\centering  % redundant
\includegraphics{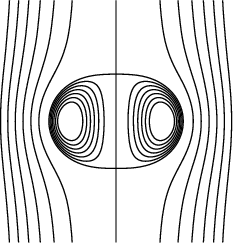}
 \centering{$\We > \gamma$}
 \end{minipage}%
 \hfill% not: "\hspace{0.5cm}"
 \begin{minipage}{0.3\linewidth}
 %\centering  % redundant
\includegraphics{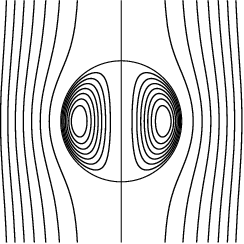}
 \centering{$\We = \gamma$}
 \end{minipage}
 \hfill% not: "\hspace{0.5cm}"
 \begin{minipage}{0.3\linewidth}
 %\centering  % redundant
\includegraphics{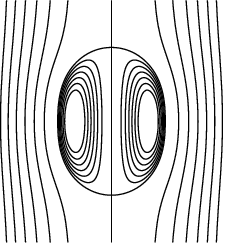}
 \centering{$\We < \gamma$}
 \end{minipage}%
 \caption{Schematic plots of the spherical vortex ($\We = \gamma$) and the perturbations for some $\gamma \notin \Gamma$ depending on the proportion of $\We$ and $\gamma$.}
 \label{fig:streamline}
 \end{figure}

 \begin{figure}[H]
\begin{minipage}{0.3\linewidth}
%\centering  % redundant
\includegraphics{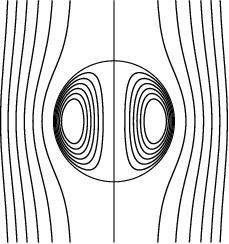}
\centering{$\We = \gamma$}
\end{minipage}%
\hfill% not: "\hspace{0.5cm}"
\begin{minipage}{0.3\linewidth}
%\centering  % redundant
\includegraphics{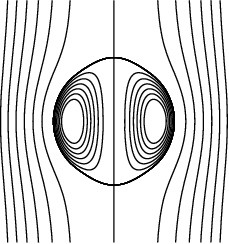}
\centering{$\weout_1$}
\end{minipage}
\hfill% not: "\hspace{0.5cm}"
\begin{minipage}{0.3\linewidth}
%\centering  % redundant
\includegraphics{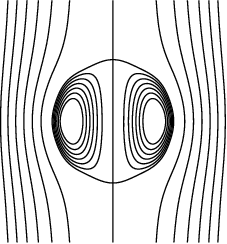}
\centering{$\weout_2$}
\end{minipage}%
\caption{Schematic plots of the spherical vortex ($\We = \gamma$) and of the bifurcations at $\weout_1$ and $\weout_2$.}
\label{fig:bifurc}
\end{figure}

\begin{corollary}
There exist values of $\gamma$ close to the bifurcation set $\Gamma$ for which non-spherical steady vortex solutions  with $\We=\gamma$ exist. In particular, for these values, the spherical vortex is non-unique.
\end{corollary}

This result is in stark contrast with what is known in the one-fluid setting, where no surface tension is present. There, Hill's spherical vortex solution is unique in the sense that for any given speed $V$ and vorticity $a$, the only solution to the corresponding elliptic problem, namely \eqref{15}, \eqref{eq:vorticity}, and \eqref{eq:jumpcond} with $\rhoin=\rhoout$ and $\sigma=0$, is a translation of  Hill's vortex, see \cite[Theorem 1.1]{Amick1986-al}.

\begin{remark} \label{weber number}
    Some remarks on Theorem \ref{thm:bubbles} are in order. 
    \begin{enumerate}
        \item For infinite surface tension, we deduce from the jump condition \eqref{eq:jumpcond} that the only admissible solution is a surface of constant mean curvature, which in $\R^3$ is a sphere. 
        Our asymptotics \eqref{eq:largasym} moreover indicate that any non-spherical steady solution near the sphere approaches the spherical shape in the limit $\sigma\to \infty$. 
        \item The asymptotics that we find for the shape function in the limit of small Weber numbers \eqref{eq:largasym} agrees with the formal predictions made by Pozrikidis \cite{Pozrikidis_1989} in the case of negligible inner circulation or density, $\gamma=0$. In this situation, we are concerned with a \emph{hollow vortex}. 
        \item To classify the shape of our spheroidal objects, we compute the ratio of the cross-stream axis extension to the parallel axis extension,
        \begin{align*}
            \chi = \frac{R+d_{\eps}(\pi/2)}{R+d_\eps(0)} = 1 + \frac{9}{32}\eps(\dout-\din) +o(\eps),
        \end{align*}
        as $\eps\ll1$.  Hence, the spheroid is oblate if $\dout>\din$ and prolate if $\dout<\din$. This observation is consistent with the physics literature on bubbles and drops \cite{Harper1972-lf}. The resulting shapes are visualized in Figure \ref{fig:streamline}.
        
        In particular, if the inner circulation or density is negligible as for bubbles, $\gamma=\din=0$, and the Weber number is small $\We=\eps\dout\ll1$, the surface $\S$ is always oblate, and we have to leading order
        \[
        		  \chi =  1+ \frac{9}{32} \We + \mathcal{O}(\We^2),
        \]
        as predicted, for instance, in \cite{Moore1959-gv,Harper1972-lf} for high Reynolds number flows. Actually, as pointed out by Harper \cite[p.~77]{Harper1972-lf} in spite of the limitation of this statement to small Weber numbers, ``oblate spheroids are found to be a fair approximation to the true shapes of bubbles for quite large values of $\We$.''
       \item Explicitly calculating the critical Weber numbers $(\wein_k)_{k \in \N}$ in Theorem \ref{thm:bubbles} seems to be a hard problem. It is related to the spectrum of an infinite Jacobi matrix. Here, an analogy to calculating the spectrum of a discrete Schr\"odinger operator can be drawn, which is only known in very special cases.
        
       Numerically, we obtain the first values as listed in Table \ref{tab:gamma_k}, which are plotted in Figure \ref{fig:seq}. The first two critical Weber numbers match the prediction in \cite{Pozrikidis_1989}, where the Weber number is defined with a factor of 2 difference from ours. The rigorous lower bound on $\wein_1$ in Theorem \ref{thm:bubbles} is not optimal.
        \item  As we cannot explicitly calculate the values $(\wein_k)_{k \in \N}$ nor the corresponding solutions of the linear problem we cannot obtain the asymptotic of the bifurcation curve constructed in Theorem \ref{thm:bubbles} explicitly. We refer to Remark \ref{rem:asymptoticG2} for details. Relying on numerical approximations we are able to provide the schematics in Figure \ref{fig:bifurc} for the first two bifurcations.
      \item It is expected that outside the perturbative setting, traveling wave solutions might not exist (or are not physically relevant). For instance, in the case of very large air bubbles (corresponding to $\weout\gg1$ and $\gamma=0$), experiments show that the fluid flow becomes unsteady and turbulent, see for instance \cite{wegener1973spherical}.
  
        \item We finally remark that for small $\weout$, the main effect in the equations is due to the mean curvature. In principle, one could also prove an analogous version of the theorem for small Weber number and small internal circulation for other interior vorticity distributions (or even other governing equations, as long as the stationary equations for a fixed domain are solvable and behave well under perturbations of the boundary). 
    \end{enumerate}
\end{remark}

\begin{table}[H]
\centering
\begin{tabular}{|c|c|c|c|c|c|c|c|c|c|c|}
\hline
$k$ & 1 & 2 & 3 & 4 & 5 & 6 & 7 & 8\\ \hline
$\wein_k$ & 2.20516 & 3.07529 & 3.94492 & 4.81679 & 5.69137 & 6.56836 & 7.44739 & 8.32829 \\ \hline
% $k$ & 1 & 2 & 3 & 4 & 5 & 6 & 7 & 8 & 9 & 10\\ \hline
% $\wein_k$ & 2.20516 & 3.07529 & 3.94492 & 4.81679 & 5.69137 & 6.56836 & 7.44739 & 8.32829 & 9.21838 & 10.18664\\ \hline
\end{tabular}
\caption{Table of values of $\gamma_k$ for $k=1, \ldots, 8$ approximated numerically.}
\label{tab:gamma_k}
\end{table}

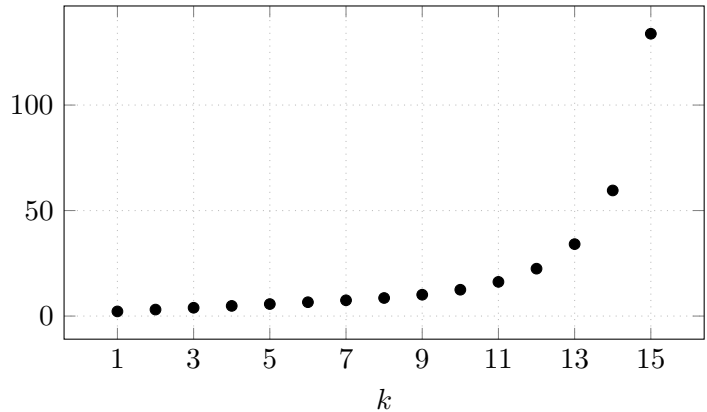
\begin{figure}[H]
         \begin{tikzpicture}
    \begin{axis}[
        width=10cm, height=6cm,
        xlabel={$k$},
        xtick={1,3,5,7,9,11,13,15},
        grid=both,
        grid style={dotted},
        yticklabel style={/pgf/number format/fixed},
        legend style={at={(0.5,-0.15)}, anchor=north, legend columns=-1}
    ]
        \addplot[
            mark=*, only marks
        ] coordinates {
            (1, 2.20516)
            (2, 3.07529)
            (3,  3.9449)
            (4, 4.81679)
            (5, 5.69137)
            (6, 6.56927)
            (7,  7.47469)
            (8,8.56158)
            (9, 10.1182 )
            (10, 12.4794)
            (11, 16.2017)
            (12, 22.4579)
            (13,34.0913)
            (14,59.524 )
            (15,133.795)
        };
    \end{axis}
\end{tikzpicture}
\caption{Values of $\gamma_k$, $k = 1,\dots,15$ approximated numerically.}
\label{fig:seq}
\end{figure}

\subsection*{Notation}\label{notation}
By $\N$ we denote the positive integers and $\N_0 = \N \cup \{ 0 \}$. In the following we write $x = (x_1,x_2,x_3) \in \R^3$ for Euclidean coordinates, $(r,z,\varphi) \in  [0,\infty) \times \R  \times [0,2\pi)$ for cylindrical coordinates and $(s,\theta,\varphi)\in [0,\infty) \times [0,\pi) \times [0,2\pi)$ for spherical coordinates, defined such that
\begin{equation*}
(x_1,x_2,x_3)=(r\cos\varphi,r\sin\varphi,z)=(s\sin(\theta)\cos(\varphi),s\sin(\theta)\sin(\varphi),s\cos(\theta)).
\end{equation*}
A function $f \colon \R^3 \to \R^3$ can be written in euclidean $f = f(x_1,x_2,x_3)$, cylindrical $f = f(r,z,\varphi)$ or spherical coordinates $f = f(s,\theta,\varphi)$. With $e_1,e_2,e_3$, $e_r,e_z,e_\varphi$ and $e_s,e_\theta,e_\varphi$ we denote the corresponding unit vectors, which form an orthonormal frame. Moreover, we write $B = B_1(0)$. The sphere of radius $R>0$ in $\R^3$ is $\Sp^2_R$. 

\section{The perturbative Ansatz}\label{sec3}

We aim to study perturbations of the spherical solution \eqref{eq:psiHill}, whose surfaces are close to the sphere of radius $R$. 
In order to non-dimensionalise the overdetermined problem \eqref{15}, \eqref{eq:vorticity},\eqref{eq:jumpcond}, we start by noticing that the change of variables
\begin{equation*}
    x=R \hat x,\quad \psi = R^3\hat \psi, \quad V = R^2\hat V,\quad \sigma  = R^5\hat \sigma,
\end{equation*}
allows us to restrict our attention to the case of the unit sphere, $R=1$. We may furthermore decompose the vector stream function into its inner and outer contributions: In the interior, we introduce  $\psiin \colon \Bin \to \R^3$ satisfying
\begin{equation} \label{eq:psiin}
    \begin{cases}
	   -\Delta \psiin = \frac{15}2 s \sin \theta \, e_\varphi &\mbox{ in } \Bin,\\[2mm]
	   \psiin =0&\mbox{ on } \S,
    \end{cases}	
\end{equation}
and in the outer domain, we let $\psiout \colon \Bout \to \R^3$ be the solution of
\begin{equation} \label{eq:psiout}
    \begin{cases}
	   -\Delta \psiout = 0 &\mbox{ in } \Bout,\\[2mm]
		\psiout = \frac{1}{2} s \sin \theta \, e_\varphi&\mbox{ on } \S,
    \end{cases}	
\end{equation}
vanishing at infinity. Regarding the well-posedness of the Laplace equation in exterior domains, we refer the reader to \cite{Amrouche1997-rb}. We note that in the present article, we only consider a small perturbation of the exterior unit ball where we obtain well-posedness by means of the Kelvin transform in the proof of Proposition \ref{P1}.

With the above notation, we write
\begin{equation*}
	\psi = \left(a\psiin+\frac{V}{2}s \sin \theta \, e_\varphi \right)\mathds{1}_{\Bin}+ V\psiout \mathds{1}_{\Bout},
\end{equation*}
for any solution to  \eqref{15},\eqref{eq:vorticity} where we consider $\S$ (and thus $\Bin$ and $\Bout$) to be given.

In terms of the new inner and outer stream functions, the  jump condition \eqref{eq:jumpcond} then reads 
\begin{equation} \label{18}
    \frac{\wein}{2}|\curl\psiin|^2-\frac{\weout}{2}|\curl\psiout-e_3|^2+H= \const \quad \mbox{ on } \S,
\end{equation}
where $\wein$ and $\weout$ denote the  Weber numbers  introduced in  \eqref{40}. If $\S$ is a sphere, then we recall from \eqref{19} that the spherical function in \eqref{eq:psiHill} is a solution to this precisely if $\wein  = \weout$.

As described in the introduction, we pursue a perturbative ansatz. We consider symmetric and Sobolev regular shape functions $\eta\in \H^{\beta}(\mathbb{S}^2)$ with $\beta \in [0,\infty)$, see Section \ref{sec:sobolev} for the definition of these spaces, and write
\begin{equation*}
	\S=\S_\eta = \left\{ (1+\eta(x))x : x \in \Sp^2 \right\}
\end{equation*}
for the graph of $\eta$ over $\Sp^2$.  We recall that by a \emph{symmetric} function, we understand a function that depends only on the polar angle  $\eta = \eta(\theta)$ and that is reflection invariant across the reference plane, $\eta(\pi/2-\theta)=\eta(\pi/2+\theta)$, cf.~\eqref{43}. Sobolev spaces on the sphere will be explained in Subsection \ref{sec:sobolev}. If $\eta$ is continuous and $\eta>-1$, the surface $\S_{\eta}$ is the boundary of a simply connected bounded set $\Bin_{\eta}$.
We shall suppose that its volume  is identical to that of the unit ball,
\begin{equation}\label{42}
    |\Bin_{\eta}| = \frac43 \pi.
\end{equation}
The set of admissible \emph{small} shape functions is thus given by 
\[
\M^{\beta} = \M^{\beta}_{c_0} = \left\{\eta\in \H^{\beta}(\Sp^2)\mbox{ symmetric}: \eqref{42} \mbox{ holds and } \|\eta\|_{\H^{\beta}(\Sp^2)}\le c_0\right\},
\]
for some suitably chosen small constant $c_0$.

We furthermore denote by   $\Bout_\eta $ the set outside  of $\S_{\eta}$ and by $H_\eta \colon \S_\eta \to \R$ the mean curvature of $\S_\eta$. We indicate the dependence of $\psiin$ and $\psiout$ on $\eta$ by writing $\psiin_\eta$ and $\psiout_\eta$, respectively. Finally, we write $\chi_\eta \colon \Sp^2 \to \S_\eta$ for the parametrization of the interface,
\begin{equation*}
    \chi_\eta(x) = (1+\eta(x))x,
\end{equation*} 
for any $x \in \Sp^2$. Recall that at any point $x \in \Sp^2$, the vector $x$ is the outer unit normal vector. 

Pulling the jump equation \eqref{18} back to the unit sphere, its left-hand side turns into the functional  
\begin{align}
\F(\wein,\weout,\eta) = \frac{\wein}2 \abs{(\curl \psiin_\eta)\circ \chi_\eta }^2 - \frac{\weout}2 \abs{(\curl \psiout_\eta)\circ \chi_\eta- e_3}^2   +   H_\eta \circ \chi_\eta.  \label{eq:G}
\end{align}
Traveling wave solutions then correspond to configurations for which $\mathcal{F}$ is constant,
\begin{equation}
    \label{41}
\F(\wein,\weout,\eta)=\const,
\end{equation}
see \eqref{eq:jumpcond}. We will study this identity in suitable function spaces in which we mod out constants. Then, finding traveling wave solutions will correspond to constructing zeros of $\mathcal{F}$.
We will occasionally write
\begin{equation} \label{eq:G2}
\F(\wein,\weout,\eta)  =  \mathcal{J}(\wein,\weout,\eta)+  \mathcal{C}(\eta),
\end{equation}
in which $\mathcal{J}$ denotes the quadratic jump term and $\mathcal{C}$ denotes the curvature term.

Clearly, if $\wein=\weout$ and $\eta=0$, the jump term is vanishing by \eqref{exp jump} and \eqref{19}, and thus, for any value of $\wein=\weout$ we recover the spherical vortex \eqref{eq:psiHill} with
\begin{equation}
\F(\wein ,\wein, 0) = \cC(0)=2 \quad  (= \const).\label{trivial 0}
\end{equation}
In Theorem \ref{thm:bubbles}, we analyze the bifurcation of \eqref{trivial 0} with respect to the parameter $\wein$: We identify a sequence of points $\Gamma = (\wein_k)_{k \in \mathbb{N}}$ through which a bifurcation curve passes along which non-spherical solutions of \eqref{41} with $\We = \gamma$ exist. The bifurcation is found by an application of  Crandall and Rabinowitz's bifurcation theorem \cite{Crandall1971-wb}. Complementary to this,  we may invoke the implicit function theorem away from the bifurcation points $\Gamma$ to construct a family of non-spherical solutions having non-identical parameters, $\wein\not=\weout$. Near the limiting case $\wein=\weout=0$, we obtain the leading order asymptotic for the shape function.
\medskip

The precise version of Theorem \ref{thm:bubbles} reads as follows:
 \begin{theorem}
     \label{T2}
Let $\beta>2$. There exists $c_0 = c_0(\beta)>0$ and a universal increasing sequence $\Gamma=(\gamma_k)_{k\in\N}$ of positive numbers diverging to infinity as $k\to \infty$ with the following property:
\begin{enumerate}
    \item For any  $\gamma\in [0,\infty) \setminus \Gamma$ and any $\weout$ close to but different from $\gamma$, there exists a unique nontrivial solution $\eta = \eta(\wein,\weout)\in \M^{\beta}_{c_0}$ to the jump equation \eqref{41}. This solution is smooth.
    \noindent
    Moreover,  if $\wein=\eps\din$ and $\weout=\eps\dout$ for two nonnegative constants $\din\not=\dout$ and a small parameter $\eps$, we have the asymptotic expansion
    \[
    \eta_{\eps} = \eps\frac{3}{32}(\din-\dout)\left(3\cos^2\theta-1\right)+o(\eps),
    \]
    as $\eps\to0$.
    \item For any $k\in\N$, there exists a unique local curve $s\mapsto \wein(s)$ passing through $\gamma_k$ and there are associated nontrivial shape functions $\eta(s)\in\M_{c_0}^\beta$ such that the jump equation \eqref{41} is solved with Weber numbers $(\wein(s),\wein(s))$. These shape functions are smooth.
\end{enumerate}
Furthermore, we have the explicit lower bound $\gamma_1 \ge \frac{60060}{16510+ 2574 \sqrt{10}+945 \sqrt{65}} \approx 1.861 \ldots$.
\end{theorem}

In view of the perturbative setting presented in the present section, the above theorem readily implies Theorem \ref{thm:bubbles}. Notice that our rescaling implies that $d_{\eps}  =R\eta_{\eps}$.

\medskip

Our analysis is structured as follows. We introduce the correct functional analytic setup in Section \ref{sec:shape}, where we also prove the Fr\'echet differentiability of $\J$ and $\cC$ and give formulas for their derivatives. In Section \ref{sec:proofs} we provide the proof of Theorem \ref{T2}. 

\section{Analysis of the functional $\F$}
\label{sec:shape}

In this section, we study the regularity properties of the functional $\F$ and derive and analyze its linearization. First,  we introduce the underlying function spaces for functions $\eta \colon \Sp^2 \to \R$ and recall some important facts on Sobolev spaces on the sphere.

\subsection{Sobolev spaces on the sphere.}
\label{sec:sobolev}

We denote by $\L^2(\mathbb{S}^{2})$ the space of square-integrable functions on the sphere equipped with the uniform measure $\dx \sigma(x) = \sin(\theta) \dx \varphi \dx \theta$ (which is the same as the Hausdorff measure on $\mathbb{S}^2$ up to a constant factor) and we write $\langle \cdot, \cdot \rangle$ for the induced $\L^2$ scalar product. 

Our analysis is based on the fact that spherical harmonics  $\{ Y_l^m : l\in \N_0, \, -l \le m \le l \} $ form an orthonormal eigenbasis of the Laplace--Beltrami operator on the sphere $\Delta_{\Sp^2}$ with respect to the $\L^2(\Sp^2)$ scalar product. The corresponding eigenvalues $-l(l+1)$ have the multiplicity $2l+1$. We have the expansion
 \begin{equation}
f=\sum_{l=0}^{\infty} \sum_{m=-l}^{l} \langle f, Y_l^m \rangle  Y_{l}^m\label{eq:orthbasis},
\end{equation}
for any $f \in \L^2(\mathbb{S}^{2})$. We recall that  spherical harmonics can be expressed as
\begin{equation} \label{eq:spherical}
Y_l^m(\theta,\varphi) = c_{l,m} P_{l}^m(\cos \theta) e^{i m \varphi},
\end{equation}
where $c_{l,m}= \sqrt{\frac{(2 l+1)}{4 \pi} \frac{(l-m)!}{(l+m)!}}$ are positive constants and $P_l^m$ are the associated Legendre polynomials. We refer the reader e.g.\ to \cite{Muller2014-up} for background reading.
 
For $\beta>0$, we define the Sobolev space $\H^\beta(\mathbb{S}^{2})$  as the space of all functions $f \in \L^2(\mathbb{S}^{2})$  
with
\begin{equation*}
	\|f\|_{\H^\beta(\mathbb{S}^{2})}^2=\sum_{l = 0}^{\infty} \sum_{m=-l}^{l}\left(1+l\right)^{2\beta}\left|\langle f, Y_l^m \rangle\right|^2<\infty.
\end{equation*}
These spaces can be equivalently defined via smooth charts \cite{Grosse2013-rl}, and thus, they   arise as the trace spaces of  $\H^{\beta+\frac{1}{2}}(B_1(0))$, cf.~\eqref{20}. For integer exponents $\beta\in\N$, they coincide with the classical Sobolev spaces defined via differentiation on the manifold. 

For notational convenience, we introduce a subspace of  $\H^\beta(\Sp^2)$ that reflects the symmetric setting we restrict to in \eqref{43}.

\begin{definition} \label{def:Hsym}
Let $\beta \ge 0$. We define 
\begin{equation*}
	 \H^\beta_{\sym}(\Sp^2):=\left\{ f \in \H^\beta(\Sp^2) : f = f(\theta) \mbox{ with } f\left( \frac{\pi}{2}-\theta \right) = f\left(\frac{\pi}{2}+\theta \right)  \right\}  ,
	\end{equation*}
	the subspace of all axisymmetric functions in $\H^\beta(\Sp^2)$, which are symmetric in $x_3$. 
\end{definition}

The following characterization will be beneficial for our analysis.

\begin{lemma}\label{lem:charX}
	For all $\beta \ge 0$ we have
    \begin{align}
	\H^\beta_{\sym}(\Sp^2) := \left\{ f \in   \H^\beta(\Sp^2) : \langle f, Y^m_l \rangle = 0 \mbox{ if $l$ is odd or } m \not= 0 \right\}. \label{eq:defX}
\end{align}
\end{lemma}

\begin{proof}
In view of the expansion of $f$ in spherical harmonics given in \eqref{eq:orthbasis} and the representation formula \eqref{eq:spherical}, it is clear that $f$ is a function of $\theta$ alone if and only if 
\[
\la f,Y^{m}_l\ra  =0 \quad\mbox{for all }m\not=0.
\]
It is thus enough to choose $m=0$, and  $\L^2_{\sym}(\Sp^2)$ is then spanned by the Legendre polynomials $P^0_l(\cos \theta)$. The symmetry with respect to $\pi/2$, which makes $f$ an even $\pi$-periodic function, is then equivalent to requiring that all generating  Legendre polynomials are even, $P_l^0(\cos \theta)  = P_l^0(-\cos \theta)$. This is the case precisely if $l$ is even, and thus
\[ 
\la f, Y^{0}_l \ra =0\quad\mbox{for any odd $l$}.
\]
This gives the desired characterization.   
\end{proof}

The following property of $\H^\beta(\Sp^2)$ and thus of $\H^\beta_{\sym}(\Sp^2)$ is crucial. 

\begin{lemma} \label{lem:embed}
	For any $ \beta >0 $ and all $k \in \N_0$, we have the   embeddings $\H^{\beta+k+1}(\Sp^2) \hookrightarrow \C^k(\Sp^2)$ and $ \H_{\sym}^{\beta+k+1}(\Sp^2) \hookrightarrow \C^k(\Sp^2) $. In particular,  $ \H_{\sym}^{\beta+1}(\Sp^2)$ is a multiplicative algebra and closed under composition with smooth functions.
\end{lemma}

\begin{proof}
Since being an element of $\H^\beta(\Sp^2)$ or $\C^\beta(\Sp^2)$, respectively, is a local property, which can be checked on open sets that are diffeomorphic to an open set of  $\R^2$, the embeddings follow directly from the corresponding Sobolev embeddings on open sets in $\R^2$ and the Leibniz rule. Similarly, the fact that the spaces are closed under composition with smooth functions can be checked on open sets, where it is well-known, see e.g.\ \cite{brezis2001composition}.
For the strategy we refer to the proof of \cite[Theorem 2.20]{Aubin1998-xd} for the case of integer $\beta \in \N$ and \cite{Grosse2013-rl} for the tools to make it work in the fractional setting.
\end{proof}

\subsection{Differentiability and linearization}
\label{sec:curvature}

In this section, we explain the differentiability of the functional $\F$ and we compute its derivative at the sphere. Our goal is the following proposition.

\begin{proposition}
    \label{P1}
    Let $\alpha>0$ be given. There exists a constant $c_0>0$ such that the functional
    \[
    \F: \R\times \R\times \M_{c_0}^{\alpha+2} (0)\to \H^{\alpha}(\Sp^2) /_\const
    \]
    is continuously Fr\'echet differentiable. Its derivative at $(\wein,\wein,0)$ is given by
    \[
    \la \D_{\eta}\F(\wein,\wein,\eta)|_{\eta=0} ,\delta\eta\ra  = \frac92  \wein\sin \theta\,{e_{\varphi}\cdot}(2 \Id -\Lambda)(\sin \theta\, \delta \eta \, {e_{\varphi}})- \left(\Delta_{\Sp^2}+2\Id \right)\delta\eta,
    \]
    where $\Lambda$ denotes the Dirichlet-to-Neumann operator of the unit ball.
\end{proposition}

To start with, we first notice that by Lemma \ref{lem:embed}, there exists for any $\alpha > 0$ a radius $\delta_{\alpha}$ such that
\[
\norm{\eta}_\infty+\norm{\nabla_{\Sp^2} \eta}_\infty < \frac{1}{2}
\]
for any $\eta\in   B_{\delta_{\alpha}}^{ \H_{\sym}^{\alpha+2} (\Sp^2)}  (0) \subset  \H_{\sym}^{\alpha+2}(\Sp^2)$.  

We introduce our solution manifold and its tangent at the origin.

\begin{lemma}\label{lem:M}
Let $\alpha >0$ be given and $c_0\le \delta_{\alpha}$.	Then the Banach manifold 
	 \begin{equation*}
 	\M^{\alpha+2} := \M^{\alpha+2}_{c_0} = \left\{ \eta \in  B_{c_0}^{ \H_{\sym}^{\alpha+2}(\Sp^2)} (0): \abs{\Bin_\eta} = \frac{4}{3} \pi  \right\}
 \end{equation*}
	is smooth and its tangent space at $\eta = 0$ is given by 
	\begin{equation*}
		T_0\M^{\alpha+2} = \left\{ \eta \in \H_{\sym}^{\alpha+2}(\Sp^2) : \int_{\Sp^2} \eta \dx \sigma = 0\right\}. 
	\end{equation*}
\end{lemma} 

\begin{proof}
The assertion is fairly well-known. See, for instance, Proposition 3.3 in 	 \cite{Meyer2024-xx} for details.
\end{proof}

A simple calculation in differential geometry, which we omit here, shows that the pull-back of the mean curvature $H_\eta\circ\chi_\eta$ for $\eta$ suitably smooth can be written as 
\begin{equation}
  \cC(\eta)=  H_\eta \circ \chi_\eta=\frac{1}{1+\eta} \left(2\frac{1+\eta}{\sqrt{g_\eta}} - \frac{\Delta_{\Sp^2} \eta}{\sqrt{g_\eta}} -  \nabla_{\Sp^2} \frac{1}{\sqrt{g_\eta}} \cdot \nabla_{\Sp^2} \eta  \right),\label{form H}
\end{equation}
where $g_\eta = (1+\eta)^2+ \abs{\nabla_{\Sp^2} \eta}^2$. See, for instance, \cite[Section 2.2]{Pruss2016-ht} or \cite[Section 7]{eiter2023falling} (where a different sign convention is employed) for a derivation.

We provide the differentiability of the curvature term.

\begin{lemma}\label{lem:Dcurv}
	Let $\alpha > 0$. The mapping
	\begin{equation*}
		\cC \colon \M^{\alpha+2}  \to  \H_{\sym}^{\alpha}(\Sp^2)/_\const
	\end{equation*}
	 is continuously Fr\'echet-differentiable and 
	\begin{equation} \label{eq:Detacurv}
		\left. \mathrm{D}_\eta \, \cC(\eta) \right|_{\eta = 0} = - \left(\Delta_{\Sp^2}+2\Id \right)\colon T_0\M^{\alpha+2} \to  \H_{\sym}^{\alpha}(\Sp^2)/_\const.
	\end{equation}
\end{lemma}

\begin{proof}
It is elementary to check that $H_\eta\circ\chi_\eta$ inherits the symmetry properties \eqref{43} from $\eta$.
 As $\alpha > 0$, an application of Lemma \ref{lem:embed} yields that $\H_{\sym}^{\alpha+2}(\Sp^2)$ is an algebra embedding into $ \C^1(\Sp^2)$, and thus $H_\eta \circ \chi_\eta\in \H_{\sym}^{\alpha}(\Sp^2)$ is well-defined. Furthermore, the dependency of the mapping \eqref{form H} on second-order derivatives of $\eta$   is linear. 

The differentiability near $\eta=0$ can then be straightforwardly checked. See, for instance,  \cite[Lemma 2.1]{Meyer2024-xx} or \cite[Lemma 2.8]{Abels2018-ht}. The form of the derivative is an immediate consequence of the explicit formula \eqref{form H}.
\end{proof}

Next, we study the regularity of the jump term and compute its derivative. 
For further reference, start by recalling that  the curl of a vector field $A = A_se_s+A_{\theta}e_{\theta}+A_{\varphi}e_{\varphi}$ on $\R^3$  reads in spherical coordinates
    \begin{equation} \label{eq:curlspher}
    \begin{aligned}
     \curl A &= \frac{1}{s \sin \theta}\left(\frac{\partial}{\partial \theta}\left(A_{\varphi} \sin \theta\right)-\frac{\partial A_\theta}{\partial \varphi}\right) e_s\\
     &\quad +\frac{1}{s}\left(\frac{1}{\sin \theta} \frac{\partial A_s}{\partial \varphi}-\frac{\partial}{\partial s}\left(s A_{\varphi}\right)\right) e_\theta \\
        &\quad+\frac{1}{s}\left(\frac{\partial}{\partial s}\left(s A_\theta\right)-\frac{\partial A_s}{\partial \theta}\right) e_\varphi.
        \end{aligned}
    \end{equation}

\noindent The main result of this subsection is the following.

\begin{lemma} \label{lem:jump}
	Let $\alpha > 0$. For $c_0$ small enough, the jump term  $\J \colon \R\times \R\times \M^{\alpha+2}_{c_0 } \to  \H_{\sym}^{\alpha+1}(\Sp^2)$ is continuously Fr\'echet differentiable with derivative $$\mathrm{D}_\eta \J(\wein,\wein,\eta)|_{\eta = 0} \colon T_0\M_{c_0}^{\alpha+2} \to  \H_{\sym}^{\alpha+1}(\Sp^2)$$ given by
    \begin{equation} \label{eq:DJ}
        \langle \mathrm{D}_\eta \J(\wein,\wein,\eta)|_{\eta=0}, \delta \eta \rangle =  \frac92  \wein \sin \theta\,{e_{\varphi}\cdot}(2 \Id -\Lambda)(\sin \theta\, \delta \eta \, {e_{\varphi}}),
    \end{equation}
    where $\Lambda$ is the Dirichlet-to-Neumann map for the Laplacian on the unit ball in $\R^3$. In particular, $\J$ is also continuously  differentiable as a map from $\mathcal{M}_{c_0}^{\alpha+2}$ to $ \H_{\sym}^{\alpha}(\Sp^2)/_\const$.
\end{lemma}

In what follows, we tacitly assume that $c_0$ is chosen small enough so that the statement of Lemma \ref{lem:jump} applies.

\begin{proof}
For our analysis, it will be beneficial to extend the shape function $\eta$ to all of $\R^3$. For this, we consider its harmonic extension $\bar\eta$, solving $\laplace \bar \eta = 0$ in $\R^3\setminus \Sp^2$ and $\bar \eta=\eta$ on $\Sp^2$. We localize again with the help of a cut-off function $\zeta$, that we choose smooth and radially symmetric, supported in $B_2(0)$ and equal to $1$ in a neighborhood of $B_1(0)$. By an abuse of notation, we set $\eta = \zeta\bar\eta$, and we obtain  $\eta\in \H^{\alpha+\frac52}(\R^3\setminus\Sp^2)$  and 
     \begin{equation}
         \label{20}
     \|\eta\|_{\H^{\alpha+{\frac52}}(\R^3\setminus \Sp^2)} \le C \|\eta\|_{\H^{\alpha+2}(\Sp^2)},
      \end{equation}
     for some constant $C = C(\alpha)>0$, by construction and elliptic regularity estimates, see \cite[Chapter 2, Thm.\ 5.4]{lions2012non}.

  Having now a globally extended shape function, we may also extend our parametrization to all of $\R^3$ by setting
\begin{equation*}
	\chi_\eta(x) = (1+\eta(x))x,
\end{equation*}
for any $x\in\R^3$. The resulting map $\chi_{\eta}:\R^3\to \R^3$
 is a diffeomorphism  because  it is a small $\C^1$-perturbation of the identity  after possibly reducing the radius $c_0$ introduced in Lemma \ref{lem:M}. More precisely, we calculate 
\begin{equation*}
	\D \chi_\eta = (1+\eta ) \Id + x \otimes \nabla \eta ,
\end{equation*}
and thus, by the matrix determinant lemma, we have the formula
\begin{equation*}
	\det(\D \chi_\eta) = (1+\eta)^2(1+\eta+x \cdot \nabla \eta).
\end{equation*}
Apparently, the determinant is positive if $\|\eta\|_{\C^1}$ is sufficiently small. This is guaranteed by the standard (fractional) Sobolev embedding in $\R^3$ and the bound in \eqref{20}, if the constant $c_0$ in Lemma \ref{lem:M} is chosen sufficiently small.

We will now study the differentiability of the inner and outer problems separately.

\medskip

\noindent
\emph{The inner problem.} Pulling back the inner problem \eqref{eq:psiin} to the unit ball $B=B_1(0)=\chi_{\eta}^{-1}(\Bin_{\eta})$ and  setting $\phin_\eta =\psiin_\eta\circ \chi_\eta  $, we find the elliptic equation
 \begin{equation} \label{eq:phiin}
	\begin{cases}
			-\nabla \cdot \left( M_\eta \nabla \phin_\eta \right) = f_\eta & \mbox{ in } B, \\
			\phin_\eta = 0& \mbox{ on } \partial B,
	\end{cases}
\end{equation}
 where 
 \begin{align}
	M_\eta =  (\det \D \chi_\eta)  \D \chi_\eta^{-1} \D\chi_\eta^{-T}, \quad 
 f_\eta  = {\frac{15}2} ( \det D \chi_\eta ) (s\sin \theta \, e_\varphi) \circ \chi_\eta.\label{def Meta}
\end{align}
We first establish that near $\eta=0$ the mapping $\eta\mapsto \phin_{\eta}$ is  continuously Fr\'echet differentiable from $\H_{\sym}^{\alpha+2}(\Sp^2)$ to $\H^{\alpha+\frac12 }_{\sym}(B)$. Our argument is very similar to that in the proof of Lemma 4.3 in \cite{Meyer2024-xx}, where more details can be found. We consider  \begin{equation*}N(\eta,\phi) = \div (M_{\eta}\grad\phi)+f_{\eta},
\end{equation*} 
which maps $\M^{\alpha+2} \times \H^{\alpha+\frac52}_0(B) $ to $\H^{\alpha+\frac12}(B)$. Of course, $N(0,\phin_0) = 0$. The coefficients $M_{\eta}$ and $f_{\eta}$ are both continuously Fr\'echet differentiable as mappings from $\M^{\alpha+2}$ to $\H^{\alpha+\frac32}_{\sym}(B)$, and $N$ is continuously Fr\'echet differentiable in $\phi\in \H^{\alpha+\frac52}_0(B)$ with derivative $\D_{\phi}N(0,\phi)|_{\phi=\phin_0} = \laplace$. 

Because of the boundary conditions imposed on $\phi$, this is the Dirichlet Laplacian on the unit ball, which is invertible from $\H^{\alpha+\frac12}(B)$ to $\H^{\alpha+\frac52}_0(B)$ (cf. \cite[Chapter 2, Thm.\ 5.4]{lions2012non}). The implicit function theorem thus guarantees that the unique solution $\phin_{\eta}$ to \eqref{eq:phiin}, considered as a function $\eta\mapsto \phin_{\eta}$ from $\M^{\alpha+2}$ to $ \H^{\alpha+\frac52}_0(B)$, depends continuously Fr\'echet differentiable on $\eta$ after possibly decreasing $c_0$. 

By the previous argument, the trace estimate \cite[Chapter 1, Thm.\ 9.4]{lions2012non}, and the chain rule,  $\eta\mapsto (\curl\psiin_{\eta})\circ\chi_{\eta} = (\curl (\phin_{\eta}\circ\chi_{\eta}^{-1}))\circ \chi_{\eta} $ is continuously Fr\'echet differentiable as a mapping from $\H^{\alpha+2}(\Sp^2)$ to $\H^{\alpha+1}(\Sp^2)$ near $\eta=0$. In particular, by using the algebra property, cf.~Lemma \ref{lem:embed}, the mapping
\begin{equation*}
	\J^{\mathrm{in}}_\eta \colon \M^{\alpha+2}  \to  \H_{\sym}^{\alpha+1}(\Sp^2) /_\const , \quad \J^{\mathrm{in}}_\eta = \frac12\abs{(\curl \psiin_{\eta})\circ\chi_{\eta} }^2
\end{equation*}
is well-defined and continuously Fr\'echet differentiable.

Let us now compute the derivative. As a preparation, we study the limiting problem. By comparison with the vector stream function for the spherical vortex \eqref{eq:psiHill}, we must have  
\begin{equation}
\label{eq:psiin0}
\psiin_0   = {\frac34}(1-s^2)s\sin\theta \, e_{\varphi},
\end{equation}
and thus, its normal trace is given by
\begin{equation}
\label{eq:dspsiin0}
\left.\partial_s \psiin_0\right|_{s=1} = -{\frac32}\sin\theta\, e_{\varphi}.
\end{equation}
We furthermore compute the  curl (see \eqref{eq:curlspher} for its representation in spherical coordinates),
\begin{equation}
    \label{eq:curlpsiin0}
\curl \psiin_0 =\left({\frac32}(1-s^2)\cos \theta\right)e_s +\left({\frac32}(2s^2-1)\sin\theta\right)e_{\theta}.
\end{equation}
This expression simplifies at the boundary, and we find that
\begin{equation}
    \label{eq:curlpsiin0s1}
\left.\curl \psiin_0\right|_{s=1} = {\frac32}\sin\theta \, e_{\theta}.
\end{equation}
%-
%
Differentiating $\mathcal{J}^{\text{in}}_{\eta}$ and using the identity \eqref{eq:curlpsiin0s1} then yields via the chain rule
\begin{align}\label{eq:DJineta}
    \la \left.\D_{\eta}\right|_{\eta=0} \mathcal{J}^{\text{in}}_{\eta},\delta \eta\ra = {\frac32}\sin\theta\la \left.\D_{\eta}\right|_{\eta=0} \left((\curl_{\theta}\psiin_{\eta})\circ\chi_{\eta}\right),\delta \eta\ra .
    \end{align}
    Here, the indexed $\theta$ in $\curl_\theta$ indicates that we are concerned with the $e_\theta$ component of the curl. In the following, we simplify the notation by omitting the evaluation $\eta = 0$ after the derivative symbol, while keeping in mind that all terms are still evaluated at this value.
    
In order to identify the right-hand side, we start by observing that for every $f$ we have
    \begin{equation}
\label{eq:Dchi}
    \D_{\eta}(f\circ\chi_{\eta}) = x\cdot \grad_x f  = s\partial_s f.
\end{equation}
Applying now the chain rule and this formula yields
\begin{equation}\label{eq:Dcurlpsiin}
\la  \D_{\eta}  \left((\curl_{\theta}\psiin_{\eta})\circ \chi_{\eta}\right),\delta \eta\ra  = \curl_{\theta}\la \D_{\eta}\psiin_{\eta},\delta \eta\ra +s\partial_s (\curl_{\theta} \psiin_0)\delta \eta.
\end{equation}
The second term can be explicitly computed. Using the formula for the curl in \eqref{eq:curlpsiin0}, we calculate
\begin{align*}
    s\partial_s (\curl_{\theta}\psiin_0)  = s\partial_s \left({\frac32}(2s^2-1)\sin\theta\right)  ={6} s^2 \sin\theta.
\end{align*}
    Evaluation at the boundary thus gives
    \begin{equation}\label{eq:sDscurlpsiin0}
\left.s\partial_s (\curl_{\theta} \psiin_0)\delta \eta\right|_{s=1} =  {6}\sin\theta\delta \eta .
    \end{equation}
For the first term in \eqref{eq:Dcurlpsiin}, we differentiate the elliptic problem \eqref{eq:psiin} for $\psiin_{\eta}$ and find thanks to \eqref{eq:Dchi} and \eqref{eq:dspsiin0},
\[
\laplace \la \D_{\eta}\psiin_{\eta},\delta \eta\ra = 0\quad \mbox{in }B,\quad \la \D_{\eta}\psiin_{\eta},\delta \eta\ra  = {\frac32} \sin\theta \, \delta \eta e_{\varphi } \quad\mbox{on }\partial B.
\]
Using this information, we find via \eqref{eq:curlspher}
\begin{equation}\label{eq:curlDpsiin_2}
\left.\curl_{\theta} \la \D_{\eta}\psiin_{\eta},\delta \eta\ra \right|_{s=1} = -\left.\frac{\partial}{\partial s}\left(s\la \D_{\eta}\psiin_{\eta},\delta \eta\ra\cdot e_{\varphi}\right)\right|_{s=1} = -{\frac32}{e_{\varphi}}\cdot(\Lambda+\Id)\left(\sin\theta \, \delta \eta \, {e_{\varphi}}\right),
\end{equation}
where $\Lambda$ denotes the Dirichlet-to-Neumann operator associated with the unit ball defined as $\Lambda g=\partial_s f$, where $f$ solves 
\[
\laplace f = 0\quad \mbox{in }B,\quad f=g\quad \mbox{on }\partial B.
\]
  Substituting now \eqref{eq:curlDpsiin_2} and \eqref{eq:sDscurlpsiin0} into \eqref{eq:Dcurlpsiin} and using the expression in \eqref{eq:DJineta}, we arrive at  
\begin{equation}\label{eq:DJineta_2}
\la  \D_{\eta}  \J^{\text{in}},\delta \eta\ra = {\frac{9}4}\sin\theta \, {e_{\varphi}\cdot}(3\Id-\Lambda ) (\sin\theta \, \delta\eta {e_{\varphi}}) .
\end{equation}

\medskip

\noindent \emph{The outer problem.} Now, we consider the pullback of the outer elliptic problem \eqref{eq:psiout} to the outer domain $\bar B^c = \chi_{\eta}^{-1}(\Bout_{\eta})$. The vector field  $\phout_\eta = \psiout_\eta  \circ \chi_\eta$ is the decaying solution to the outer domain problem 
\begin{equation} \label{eq:phiout}
	\begin{cases}
			-\nabla \cdot \left( M_\eta \nabla \phout_\eta \right) = 0 & \mbox{ in } \bar B^c, \\
			\phout_\eta = h_\eta & \mbox{ on } \partial B,
	\end{cases}
\end{equation}
where 
\[
 h_\eta = \frac{1}{2}(s\sin \theta \, e_\varphi) \circ \chi_\eta,
\]
and $M_\eta$ was defined in \eqref{def Meta}.

{Establishing the differentiability of the solution for an inner domain problem is slightly easier than establishing it for an outer domain problem. We thus reflect \eqref{eq:phiout} to a problem on the unit ball $B$ by executing a Kelvin transformation: We consider
\[
\tilde x = \frac{x}{|x|^2},\quad \tilde \varphi_{\tilde \eta}(\tilde x) =  |x| \varphi_{\eta}(x),\quad  \tilde \eta(\tilde x) = \eta(x),\quad  \tilde h_{\tilde \eta}(\tilde x) = h_{\eta}(x),
\]
and obtain the inner domain problem
\begin{equation}
    \label{44}
\begin{cases}
    -\tilde \nabla \cdot\left( \tilde M_{\tilde \eta} \tilde \nabla \tilde \varphi_{\tilde \eta}^\text{out} \right)+  \frac{\tilde x}{|\tilde x|^2}\cdot \tilde M_{\tilde \eta}\tilde \grad\tilde \varphi^{\text{out}}_{\tilde \eta}+  |\tilde x|\tilde \nabla\cdot\left(\frac1{|\tilde x|^3} \tilde m_{\tilde\eta} \tilde x \tilde \varphi^{\text{out}}_{\tilde \eta}\right) =0&\mbox{ in }B,\\[2mm]
    \tilde \varphi^{\text{out}}_{\tilde \eta} =\tilde h_{\tilde \eta}&\mbox{ on }\partial B,
\end{cases}
\end{equation}}
\noindent with
\begin{align*}
\tilde M_{\tilde \eta}(\tilde x)  &= \left(\Id-2\frac{\tilde x}{|\tilde x|}\otimes \frac{\tilde x}{|\tilde x|}\right)M_{\eta}\left(\frac{\tilde{x}}{|\tilde x|^2}\right)\left(\Id-2\frac{\tilde x}{|\tilde x|}\otimes \frac{\tilde x}{|\tilde x|}\right) ,\\[2mm]
\tilde m_{\tilde \eta}(\tilde x) & = \left(\Id-2\frac{\tilde x}{|\tilde x|}\otimes \frac{\tilde x}{|\tilde x|}\right) M_{\eta}\left(\frac{\tilde x}{|\tilde x|^2}\right).
\end{align*}
Notice that for $\tilde \eta=0$, the matrix in the leading order term is the identity, $\tilde M_0  =\Id$ as $(\Id-2\frac{x}{|x|}\otimes\frac{x}{|x|})^2=\Id$, and the lower order terms cancel out,  i.e.,
\begin{equation*}
\frac{\tilde x}{|\tilde x|^2}\cdot \tilde M_{\tilde \eta}\tilde \grad\tilde \varphi^{\text{out}}_{\tilde \eta}+  |\tilde x|\tilde \nabla\cdot\left(\frac1{|\tilde x|^3} \tilde m_{\tilde\eta} \tilde x \tilde \varphi^{\text{out}}_{\tilde \eta}\right)=0 \quad \text{ for $\tilde\eta=0$.}
\end{equation*}
 As $\tilde \eta=0$ inside a ball of radius $1/2$, the elliptic problem is the Laplacian in the ball of radius $1/2$ and a small and regular perturbation of the limiting Laplace equation $ \tilde \Delta \tilde \varphi^{\text{out}}_0=0$.

Differentiability is now proved analogously to the homogeneous problem \eqref{eq:phiin}. This time, we consider the functional 
\[
N(\eta,\tilde \phi) = \left(-\tilde \nabla \cdot \left( \tilde M_{\tilde \eta}\tilde  \nabla \tilde \phi\right)+\frac{\tilde x}{|\tilde x|^2} \cdot \tilde M_{\tilde \eta} \tilde \grad\tilde \phi + |\tilde x| \tilde \nabla\cdot\left(\frac1{|\tilde x|^3}\tilde m_{\tilde \eta}\tilde x\tilde \phi\right), \tilde \phi|_{\Sp^2}-\tilde h_{\tilde \eta}\right).
\]
It is well-defined and continuously Fr\'echet differentiable as a mapping from $\M^{\alpha+2}\times \H^{\alpha+\frac52}(B)$ to $\H^{\alpha+\frac12}(B)\times \H^{\alpha+2}(\Sp^2)$. Moreover, because $N(0,\tilde \varphi_0^{\text{out}}) = (0,0)$ and 
\begin{equation*}
    \langle \D_{\tilde \phi} N(0,\tilde\phi)|_{\tilde \phi = \tilde \varphi_0^{\text{out}}}, \delta \tilde \phi \rangle = (-\Delta \delta \tilde \phi, \delta\tilde \phi|_{\Sp^2}) \colon \H^{\alpha+\frac52}(B) \to \H^{\alpha+\frac{1}{2}}(B) \times \H^{\alpha+2}(\Sp^2).
\end{equation*} 
This mapping is a diffeomorphism, and thus, by the implicit function theorem, for small $\eta$, the unique solution to \eqref{44} depends continuously Fr\'echet differentiable on $\tilde \eta$. Transforming back to the original problem, we get the desired result for $\phout.$

\medskip

From here, we deduce that
\begin{equation*}
	\J^{\mathrm{out}}_\eta \colon \M^{\alpha+2}  \to   \H_{\sym}^{\alpha +1}(\Sp^2) /_\const  , \quad \J^{\mathrm{out}}_\eta = \frac12\abs{\curl \phout_\eta-We_3}^2,
\end{equation*}
is well-defined and continuously Fr\'echet differentiable near $\eta=0$.

We compute the derivative. Since most of the computations are similar to the inner problem, we will present a more concise explanation.
Again, we start by analyzing the limiting function. In analogy to the spherical vortex \eqref{eq:psiHill}, we find that
\[
\psiout_0   =  \frac{1}{2s^2} \sin\theta \,  e_{\varphi}
\]
by \eqref{eq:psiHill}, an application of the curl operator (cf.~\eqref{eq:curlspher}) yields
\begin{equation}
    \label{eq:curlpsiout0} \curl \psiout_0   =  \frac{1}{s^3} \cos \theta \, e_s + \frac12 \frac1{s^3}\sin \theta\, e_{\theta},
\end{equation}
and thus, evaluation at the boundary gives
\begin{equation}
    \label{eq:curlpsiout0s1}
\left.\curl \psiout_0 \right|_{s=1}-   e_3 = \frac32  \sin\theta \,   e_{\theta},
\end{equation}
where we have used the change of basis formula $e_3 = \cos\theta \, e_s -\sin\theta \, e_{\theta}$.
Similarly, we have for the Neumann trace
\begin{equation}
    \label{eq:Dspsiout0}
\left.\partial_s\psiout_0\right|_{s=1} = -    \sin\theta \, e_{\varphi}.
\end{equation}
Arguing analogously to the above, we compute, using \eqref{eq:curlpsiout0s1}
\begin{align*}
    \la \D_{\eta}\J^{\text{out}} ,\delta \eta\ra & = \left(\curl \psiout_0 - e_3\right)\cdot \la \D_{\eta}  (\curl \psiout_{\eta})\circ \chi_{\eta},\delta \eta\ra \\
    & = \frac32  \sin \theta\,\left(\curl_{\theta} \la \D_{\eta} \psiout_{\eta},\delta \eta\ra + s\partial_s (\curl_{\theta}\psiout_0)\delta\eta\right).
\end{align*}
From the formula \eqref{eq:curlpsiout0} for the curl, we deduce that
\[
\left.\partial_s \curl_{\theta}\psiout_0\right|_{s=1} =-\frac32 \sin \theta.
\]
For the other term, we differentiate the elliptic problem \eqref{eq:phiout} and find by applying
 \eqref{eq:Dchi} and \eqref{eq:Dspsiout0},
\[
\Delta \la \D_{\eta}\psiout_{\eta},\delta \eta\ra =0\quad \mbox{in $B^c$},\quad \la \D_{\eta}\psiout_{\eta},\delta \eta\ra =\frac32  \sin\theta \, \delta \eta e_{\varphi}\quad \mbox{on }\partial B.
\]
Using once more the formula \eqref{eq:curlspher} for the curl, we conclude that
\[
\curl_{\theta} \la  \D_{\eta}\psiout_{\eta},\delta \eta\ra  = -\frac32   {e_{\varphi}\cdot}(\Id+\Lambda^{\text{out}})(\sin\theta \, \delta \eta{e_{\varphi}}) =\frac32  {e_\varphi\cdot}\Lambda^{\text{in}} (\sin \theta\, \delta \eta{e_{\varphi}}),
\]
where $\Lambda^{\text{out}}$ is the Dirichlet-to-Neumann operator associated with the outer domain problem,
that is, $\Lambda^{\text{out}} g = \partial_s f$, if 
\[
\laplace f = 0\quad\mbox{in }B^c,\quad f=g\quad \mbox{on }\partial B,
\]
which we have rewritten in terms of the  Dirichlet-to-Neumann operator associated to the inner problem via the identity
\[
\Lambda^{\text{out}} = - \Lambda - \Id.
\]
The latter can be verified by translating the Dirichlet problem on $B^c$ via the Kelvin transform into a Dirichlet problem on $B$.

We combine all estimates, use the previous identity, and find
\[
\la \D_{\eta}\J^{\text{out}} ,\delta\eta\ra  = \frac94   \sin\theta \, {e_{\varphi}\cdot}(\Lambda-\Id)(\sin \theta\,\delta \eta \, {e_{\varphi}}).
\]

\medskip

\noindent
\emph{Conclusion.} Combining the previous two derivations and invoking the relation \eqref{19}, we arrive at
\begin{align*}
 \mel \la \left.\D_{\eta}\right|_{\eta = 0} \J (\wein,\wein,\eta),\delta \eta\ra\\
  &= {\frac{9}4}  \wein  \sin\theta \, {e_{\varphi}\cdot}(3\Id -\Lambda)(\sin \theta\, \delta \eta{e_{\varphi}})+\frac94  \wein \sin\theta \, {e_{\varphi}\cdot}(\Id -\Lambda)(\sin \theta\, \delta \eta \, {e_{\varphi}})\\
  & = \frac92 \wein  \sin \theta\,{e_{\varphi}\cdot}(2 \Id -\Lambda)(\sin \theta\, \delta \eta \, {e_{\varphi}}). \qedhere
\end{align*}
\end{proof}

\subsection{Properties of the linear operator}

We start with the discussion of the invertibility of the surface tension term. It will be crucial when analyzing the regime of large surface tensions.
 
\begin{lemma}\label{lem:deriinv}
	For $\alpha \ge 0$ the operator 
	\begin{equation*}
		-\left(\Delta_{\Sp^2}+2\Id\right) \colon T_0\M^{\alpha+2} \to  \H_{\sym}^{\alpha}(\Sp^2) /_\const
	\end{equation*}
	is an isomorphism.
\end{lemma}

\begin{proof}
	With respect to the orthonormal basis $\{ Y_l^0 : l \in \N_0, \ l \neq 1 \}$ the operator is the multiplication operator with nonzero symbol $-(l+2)(l-1)$ and thus invertible with a loss of two derivatives. Note that the zero mean condition matches the fact that we take the quotient space with respect to constant functions in the image space.
\end{proof}

We turn our attention to the full operator. For notational convenience, we introduce 
\begin{equation} \label{eq:A}
\begin{aligned}
	[\mathcal{A}(\mu)](\delta \eta) & =\frac{2}{9\wein} \langle \left.\D_{\eta}\right|_{\eta=0} \F (\wein,\wein,\eta),\delta \eta\ra \\
    & =  \sin \theta\,{e_{\varphi}\cdot}(2\Id -\Lambda)(\sin \theta\, \delta \eta \, {e_{\varphi}}) -\mu  (\Delta_{\Sp^2}+2 \Id)\delta \eta,
    \end{aligned}
\end{equation}
where we write $\mu  = 2/(9\wein)$ in the following.  For our analysis, it will be beneficial to rewrite the linear operator $\mathcal{A}$ with the help of spherical harmonics. This can be achieved by expressing the function $\delta\eta$ introduced in Lemma \ref{lem:M} in terms of spherical harmonics. According to the characterisation \eqref{eq:defX} and because $\delta \eta$ has zero mean, we have the expansion
\[
\delta \eta(\theta) = \sum_{k=1}^{\infty} v_k Y_{2k}^0(\theta),
\]
for some $v_k \in \R$. Precise properties of the coefficients will be discussed later. Next, we identify $\A(\mu)$ as an operator on sequences $(v_k)_{k \in \N}$. For $\alpha  \ge 0 $ we set
\begin{equation*}
    \h^\alpha:=\left\{v = (v_k)_{k\in \N}\, :\, \norm{v}_{\h^\alpha}^2 := \sum_{k = 1}^\infty k^{2\alpha}v_k^2<\infty\right\},
\end{equation*}
which is a Banach space. We have the compact embedding $\h^{\alpha}\hookrightarrow \h^{\alpha-\beta}$ for all $\beta>0$, and $\h^\alpha$ is isomorphic to $\H_{\sym}^{\alpha}(\Sp^2)/_\const$ and isomorphic to $\left\{ \eta \in \H_{\sym}^{\alpha}(\Sp^2)  : \int_{\Sp^2} \eta \dx \sigma = 0 \right\}$ via $(v_k)_{k\in \N}\mapsto \sum\limits_{k = 1}^\infty v_k Y_{2k}^0$. 

\begin{lemma} \label{lem:rep}
For $\alpha \ge 0$ we have that $\A(\mu) \colon \h^{\alpha+2} \to \h^{\alpha}$ is symmetric with the representation
    \begin{equation}\begin{aligned}
    &[\A(\mu)]\left((v_k)_{k \in \N}\right) := [\A(\mu)]\left(\sum_{k=1}^{\infty}v_k Y_{2k}^0\right)\\
    &= \sum_{k=2}^{\infty}\left(A_k(\mu)v_k + B_kv_{k-1}+C_kv_{k+1}\right) Y_{2k}^0 +\left(A_1(\mu)v_1 + C_1v_2\right)Y_2^0 + C_0v_1 Y_0^0, \label{def cal A}
\end{aligned}\end{equation}
where 
\begin{align*}
    A_k(\mu) &= \mu (2k-1)(2k+2) -\frac{{2}k(2k-3)(2k-1)}{(4k+1)(4k-1)} -\frac{(2k-1)(2k+1)({2}k+{2})}{(4k+1)(4k+3)},\\
    B_{k}  &= \frac{2k(2k-3)(2k-1)}{(4k-3)(4k-1)} \sqrt{\frac{4k-3}{4k+1}},\\
    C_{k}  &= \frac{(2k-1)(2k+1)(2k+2)}{(4k+3)(4k+5)} \sqrt{\frac{4k+5}{4k+1}}.
\end{align*}
\end{lemma}

\begin{proof}
    We discuss the linear operator term by term, starting with the simplest one, the curvature term. Recalling that spherical harmonics of degree $l$ are eigenfunctions of the Laplace--Beltrami operator for the eigenvalue $-l(l+1)$, we observe that
\[
(\Delta_{\Sp^2}+2\Id )Y_{2k}^0 = -  (2k-1)(2k+2) Y_{2k}^0.
\]
We now address the term that involves the Dirichlet-to-Neumann operator. We recall that the spherical harmonics can be expressed in terms of the associated Legendre polynomials, $Y_{2k}^0(\theta) = c_{2k,0}P_{2k}^0(\cos(\theta)) = c_{2k,0}P_{2k}^0(t)$ with $t=\cos\theta$, and we remark that we have the recurrence formula for these polynomials
\begin{equation}\label{eq:recursionPlm}
\sqrt{1-t^2} P^m_l(t) = (2l+1)^{-1}\left(P_{l-1}^{m+1}(t)-P_{l+1}^{m+1}(t)\right),
\end{equation}
which can be deduced from \cite[eqs.\ (8.5.1) and (8.5.3)]{Abramowitz1965-dz}. We generously identify $e_{\varphi}=ie^{i\varphi}$, and deduce
\begin{equation*}\begin{aligned}
    \sin\theta \,  Y_{2k}^0(t) \, e_{\varphi} & =   \frac{c_{2k,0}}{4k+1}   \left(P_{2k-1}^1 (\cos\theta) - P_{2k+1}^1(\cos\theta)\right)e_{\varphi}\\
    & = i  \frac{c_{2k,0}}{4k+1} \left( \frac{1}{c_{2k-1,1}}Y_{2k-1}^1(\theta,\varphi) - \frac{1}{c_{2k+1,1}}Y_{2k+1}^1(\theta,\varphi)\right).
\end{aligned}\end{equation*}
Because the functions $s^{l}Y_l^m(\theta,\varphi)$ are harmonic in $B$ (see \cite{Muller2014-up}), the spherical harmonics $Y_l^m$ are eigenfunctions of the Dirichlet-to-Neumann operator $\Lambda$ for the eigenvalue $l$. Therefore,
\begin{align*}
   (2\Id-\Lambda) (\sin\theta \, Y_{2k}^0 \, e_{\varphi} )& = i    \frac{c_{2k,0}}{4k+1}  \left(\frac{3-2k}{c_{2k-1,1}}Y_{2k-1}^1(\theta,\varphi) -  \frac{1-2k}{c_{2k+1,1}}Y_{2k+1}^1(\theta,\varphi)\right)\\
    & =   - \frac{c_{2k,0}}{4k+1}  \left((2k-3)P_{2k-1}^1(t ) - (2k-1)P_{2k+1}^1(t)\right)e_{\varphi}.
\end{align*}
Noticing that
\begin{equation*}
\sqrt{1-t^2}P_l^1(t) = \frac{l(l+1)}{2l+1} \left(P_{l+1}^0(t) - P^0_{l-1}(t)\right),%\label{leg}
\end{equation*}
which is a consequence of Equations (8.5.1) and (8.5.3) in \cite{Abramowitz1965-dz},
helps us to rewrite
\begin{align*}
\mel  \sin\theta \,   e_{\varphi} \cdot (2\Id-\Lambda) (\sin\theta \, Y_{2k}^0 \, e_{\varphi} )\\
 & =   \left(-\frac{2k(2k-3)(2k-1)}{(4k+1)(4k-1)}-\frac{(2k-1)(2k+1)(2k+2)}{(4k+1)(4k+3)}\right)  Y_{2k}^0\\
 &\quad +  \frac{(2k-1)(2k+1)(2k+2)}{(4k+1)(4k+3)} \frac{c_{2k,0}}{c_{2k+2,0}} Y_{2k+2}^0\\
 &\quad +  \frac{2k(2k-3)(2k-1)}{(4k+1)(4k-1)}\frac{c_{2k,0}}{c_{2k-2,0}} Y_{2k-2}^0.
\end{align*}
We eventually combine all the previous calculations in order to express the linear operator $\A$ in \eqref{eq:A} in terms of spherical harmonics, 
\begin{equation*}
	 \A Y_{2k}^0(t)  = A_k Y_{2k}^0(t) + B_{k+1} Y_{2k+2}^0(t) + C_{k-1} Y_{2k-2}^0(t), 
\end{equation*}
where the coefficients $A_k$, $B_k$ and $C_k$ are given as in the statement of the lemma.

Note that $B_k=C_{k-1}$, in particular, the operator $\mathcal{A}$ is symmetric.
Going back to the linear combination $\delta \eta$ and performing two index shifts gives \eqref{def cal A}.
\end{proof}

We have to analyze the kernel of this operator.

\begin{proposition}\label{prop A}
Let $\alpha \geq 0$. 
\begin{enumerate}[a)]
    \item For any $\mu\neq 0$, the operator $\A(\mu): \h^{\alpha + 2} \to \h^{\alpha}$ is a symmetric Fredholm operator of index $0$.
    \item For any $\mu>0$, the nullspace $N(\A(\mu))$ of $\A(\mu)$ is at most one-dimensional and $N(\A(\mu)) \subset \h^\beta$ for all $\beta \ge 0$. Moreover, $N(\A(\mu)) = \{0\}$ for $\mu \le 0$.
    \item There exists a strictly decreasing sequence $(\mu_k)_{k\in \N}\subset \R^+$ with limit $0$ such that $\A({\mu_k})$ has a $1$-dimensional nullspace and $A(\mu)$ is invertible if $\mu \notin \{ \mu_k : k \in \N \} \cup \{0 \}$. 
    \item We have $\mu_1 \le \frac{\sqrt{2}}{21 \sqrt{5}}+\frac{\sqrt{5}}{22\sqrt{13}}+\frac{127}{2079} \approx 0.119394$.
    \item If $0\neq v^k\in N(\A({\mu_k}))$, then the transversality condition 
    \begin{equation*}
        \mathrm{D}_{\mu}\A(\mu)\big|_{\mu=\mu_k}v^k\notin R(\A({\mu_k}))
    \end{equation*}
    holds true.
\end{enumerate}
\end{proposition}

\begin{proof}
We shall further split 
\begin{equation*}
\A(\mu)=\mu \A^1-\A^2,
\end{equation*}
where $\A^1$ and $\A^2$ denote the parts of $\mathcal{A}$ that are linear respectively constant in $\mu$. We start with observing that both $\A^1$ and $\A^2$ are symmetric operators, that is,
\begin{equation}
    \label{28}
    \la v ,\A^i w\ra  = \la \A^i v,w\ra,\quad  \mbox{for any }v,w\in \h^{\alpha+2},
\end{equation}
(with the usual $\ell^2$-scalar product) as a consequence of \eqref{def cal A}. This can also be seen on the level of functions from the symmetry of the Dirichlet-to-Neumann and Laplace--Beltrami operators. Moreover, $\A^1$ is positive as it is given by an infinite diagonal matrix with positive entries.

We will now prove the precise statements of the proposition.

\medskip
\noindent
a) The operator $\mathcal{A}^1$ is invertible and hence Fredholm with index $0$ as a map from $\h^{\alpha+2}$ to $\h^{\alpha}$. The operator $\A^2$ maps $\h^{\alpha+2}$ to $\h^{\alpha+1}$ and is hence compact with respect to $\mu\A^1$, because $\mu\not=0$ by assumption.  Therefore, $\mu\A^1-\A^2=\A(\mu)$ is a compact perturbation of a Fredholm operator of index zero and thus it is Fredholm with index $0$, too. The symmetry is stated for the individual ingredients in \eqref{28}.

\medskip
\noindent
b) To see that the nullspace for any $\mu\neq 0$ is at most $1$-dimensional, we observe that $\mathcal{A}({\mu})v=0$ is equivalent to the recurrence equation 
\begin{equation} \label{eq:rec}
    \begin{cases}
        v_{k+1}={-}\frac{A_k(\mu)}{C_k}v_k{-}\frac{B_k}{C_k}v_{k-1} \quad \text{ for $k\geq 2$},\\[2pt]
        A_1(\mu)v_1+C_1v_2=0.
    \end{cases}
\end{equation}
As $C_1\neq 0$, we note that $v_1$ determines all other numbers $v_k$ in this recursion and hence the space of solutions is $1$-dimensional and might or might not be a subspace of $\h^{\alpha+2}$. Indeed, the Poincar\'e--Perron theorem \cite[Theorem 8.35]{elyadi_difference_2005} implies that there exist exactly two linearly independent solutions of the recurrence equation without the initial condition. One of these grows of factorial order $8\mu k$ while the other decays of factorial order $-8\mu k$ as $k \to \infty$. The difficulty is to understand how to initial condition prescribed by $ A_1(\mu)v_1+C_1v_2=0$ relates to the behavior at infinity. 

Regarding the regularity statement, we notice that if $v\in \h^{\alpha+2}$ lies in the kernel of $\A(\mu)$, then
\begin{equation}\label{26}
\mu \A^1v=\A^2 v.
\end{equation}
From the regularity properties of $\A^2$ it follows that $\A^2 v\in \h^{\alpha+1}$ and as $\A^1$ is invertible from $\h^{\alpha+1}$ to $\h^{\alpha+3}$ we deduce $v\in \h^{\alpha+3}$. Iteration shows the statement.
\medskip

Next, we prove that for $\mu \le 0 $, the kernel is trivial, $N(\A(\mu)) = \{ 0 \}$. We achieve this by showing that any non-trivial sequence $v=(v_k)_{k\in\N} $ for which $[\A(\mu)](v)$ vanishes  cannot belong to $\h^{\alpha+2} $. 

Let $v$ be such an element. In view of the recurrence equations \eqref{eq:rec}, the first element $v_1$ is necessarily non-zero. Without loss of generality, we suppose that it is positive, $v_1>0$. In order to simplify the recurrence equation, we write $w_k = \sqrt{4k+1}v_k$. Using the monotonicity $A_1(\mu) <A_1(0)$, we obtain for the second element,
\[
w_2 =-\frac{3A_1(\mu)}{\sqrt{5}C_1}w_1 \ge -\frac{3A_1(0)}{\sqrt{5}C_1}w_1 = \frac{11}{10}w_1,
\]
and thus, $w_{k+1}>w_{k}>0$ for $k=1$. We aim to argue by induction and suppose that we have already established the monotonicity $w_{k+1}>w_{k}>0$ for some $k\ge1$. Using that $A_k(\mu)<A_k(0)$ for any $\mu<0$ and invoking the induction hypothesis, we deduce from the recurrence equation \eqref{eq:rec} that
\begin{align*}
w_{k+1} &\ge -\frac{A_k(0)}{C_k}\sqrt{\frac{4k+5}{4k+1}}w_k - \frac{B_k}{C_k} \sqrt{\frac{4k+5}{4k-3}}w_{k-1}\\
&\ge  - \left(\frac{A_k(0)}{C_k}\sqrt{\frac{4k+5}{4k+1}}+\frac{B_k}{C_k} \sqrt{\frac{4k+5}{4k-3}}\right)w_{k},
\end{align*}
for any $k\ge 2$. We compute and estimate the coefficient,
\begin{align*}
    \mel
-\frac{A_k(0)}{C_k}\sqrt{\frac{4k+5}{4k+1}}-\frac{B_k}{C_k} \sqrt{\frac{4k+5}{4k-3}}  =  1+  \frac{4 (40 k^2 + 38 k + 3)}{(k + 1) (2 k + 1) (4 k - 3) (4 k - 1) (4 k + 1)}>1 .
\end{align*}
and thus, $w_{k+1}>w_k$ follows.

Being a nonnegative increasing sequence, $(w_k)_{k\in\N}$ does not belong to $\h^{\alpha+1/2}$. Consequently, $(v_k)_{k\in\N}$ is not an element of $\h^{\alpha+2}$.

\medskip
\noindent
c) The above equation \eqref{26} for the kernel of $\A(\mu)$ can be rewritten as an eigenvalue problem for a self-adjoint compact operator. Indeed, because $\A^1$ is diagonal and positive definite, there exists a positive definite diagonal operator $\B$ (which is simply the square root of the diagonal matrix) such that $\A^1=\B\B$. In particular,  $v$ is in the kernel of $\A({\mu})$, cf.~\eqref{26}, precisely if $w=\B v$ is an eigenvector of the operator $\K = \B^{-1}\A^2\B^{-1}$, i.e. 
\[
\K w= \mu w.
\]
By construction and because both $\A^2$ and $\B$ are symmetric, the new operator $\K$ is symmetric as well. It is self-adjoint because it is compact (and thus, in particular, necessarily bounded): As $\B^{-1}$ maps $\h^{\alpha}$ to $\h^{\alpha+1}$ and $\A^2$ maps $\h^{\alpha+1}$ to $\h^{\alpha}$, the operator $\K$ maps $\h^{\alpha}$ to $\h^{\alpha+1}$, and it is thus compact as an operator from $\h^{\alpha}$ to $\h^{\alpha}$. Hence, by the spectral theorem for self-adjoint operators, $\K$ has countably infinite real spectrum, $\sigma(\K) = \{\mu_k\}_{k\in\N} \cup \{0 \} \subset \R$, where $\mu_k$ is a sequence of eigenvalues with limit zero. By b) we cannot have {non-positive eigenvalues and hence may assume that $(\mu_k)_{k\in\N}$ is a decreasing sequence of positive real numbers.

The statement now follows because eigenvectors are precisely the non-trivial elements of $N(\A({\mu}))$.

\medskip
\noindent
d) We employ the Gershgorin circle theorem for infinite matrices. Any $\mu_k$ constructed in c) is an eigenvalue of the infinite matrix $\K = (\K_{l,j})_{l,j \in \N}$, given by
\begin{align*}  
(\K v)_l & = -\frac{A_l(0)}{(2l-1)(2l+2)}v_l \\
&\hphantom{=}  - \frac{B_l}{\sqrt{(2l-3)2l(2l-1)(2l+2)}} v_{l-1} - \frac{C_l}{\sqrt{(2l+1)(2l+4)(2l-1)(2l+2)}}v_{l+1},
\end{align*}
for any $l\ge 1$, where we have set $B_1=0$.

Moreover, any eigenvector belongs to $\ell^1$ by b). Hence \cite[Theorem 1 (a)]{MR910985} applies and we deduce that 
\begin{equation*}
    \mu_k \in \bigcup_{l = 1}^\infty \left[\K_{l,l}-r_l,\K_{l,l}+r_l\right], \quad r_l = \sum\limits_{\substack{j = 1\\ j \neq l}}^\infty \abs{\K_{l,j}},
\end{equation*}
hence $r_l$ is the column sum. In view of the monotonicity of the eigenvalues and  the fact that the matrix is non-zero only on the diagonal and  the two off-diagonals, we then have
\[
\mu_k\le \mu_1 \le \max_{l \in \N} \left(\K_{ll} +r_{l}\right) = \max_{l \in \N} \left(   \K_{l,l} + |\K_{l,l-1}| +|\K_{l,l+1}| \right),
\]
where we have set $\K_{1,0}=0$.}
We compute
\begin{align*}
  \mel  \kappa(l):=  \K_{l,l} + |\K_{l,l-1}|   + |\K_{l,l+1}|  \\
  &= \frac1{\sqrt{(2l-1)(2l+2)}}\left(   \frac{|A_l(0)|}{\sqrt{(2l-1)(2l+2)}} + \frac{B_l}{\sqrt{2l(2l-3)}} + \frac{C_l}{\sqrt{(2l+1)(2l+4)}} \right).
\end{align*}
On the one hand, we observe that  $|A_l(0)|\le l$ and $B_l,C_l\le l/2$ for any $l\ge 2$, and thus 
\begin{align*}
\mel     \kappa(l)
     \le \frac{l}{\sqrt{(2l-1)(2l+2)}}\left(   \frac{1}{\sqrt{(2l-1)(2l+2)}} + \frac{1}{2\sqrt{2l(2l-3)}} + \frac{1}{2\sqrt{(2l+1)(2l+4)}} \right)
\end{align*}
for any $l\ge 2$. The right-hand side is a monotone decreasing function of $l$, and evaluating at $l=5$ we find $\kappa(l) <0.1$   for any $l\ge 5$. On the other hand, evaluating $\kappa(l)$ for $l\in\{1,\dots,4\}$, we find that its maximal value is attained at $l=2$, hence
\[
\mu_1 \le \kappa(2) =  \frac{\sqrt{\frac{2}{5}}}{21}+\frac{\sqrt{\frac{5}{13}}}{22}+\frac{127}{2079} \approx 0.119394.
\]

\medskip
\noindent
e) As $\A({\mu_k})$ is a Fredholm operator of index $0$ having a one-dimensional nullspace, its range has codimension $1$. We choose  $v^k\in N(\A({\mu_k}))$ and  derive from the symmetry of $\mathcal{A}$ that for all $v\in \h^{\alpha+2}$ we have
\begin{align*}
\scalar{v^k}{\A({\mu_k})v}=\scalar{\A({\mu_k})v^k}{v}=0.
\end{align*}
This shows that  $v^k$ is orthogonal to the range of $\A({\mu_k})$. However, by the positive definiteness of $\mu_k \A^1 = \mathrm{D}_{\mu}\A({\mu})|_{\mu = \mu_k}$ we have that
\begin{align*}
\scalar{v^k}{\mathrm{D}_{\mu}\A(\mu)\big|_{\mu=\mu_k}v^k}=\mu_k \scalar{v^k}{\A^1 v^k}>0,
\end{align*}
hence the vector $\mathrm{D}_\mu\A(\mu)|_{\mu = \mu_k} v^k $ cannot be in the range of $\A({\mu_k})$.
\end{proof}

\section{Proof of Theorem \ref{thm:bubbles}}
\label{sec:proofs}
The main result is a direct consequence of the implicit function theorem and the theorem of Crandall and Rabinowitz. For the readers' convenience, we provide the following calculation. Using the explicit formulas for the spherical harmonics 
\begin{equation*}
	  Y_0^0(\theta)=\frac{1}{\sqrt{4\pi}} \; \mbox{ and } \; Y_2^0(\theta)=\frac{1}{4} \sqrt{\frac{5}{\pi }} \left(3 \cos ^2(\theta)-1\right),
\end{equation*}
and the Pythagorean theorem, we find
\begin{equation}\label{45}
	 \sin^2(\theta)=\frac{4\sqrt{\pi}}{3}Y_0^0(\theta)-\frac{4\sqrt{\pi}}{3\sqrt{5}}Y_2^0(\theta).
\end{equation}
Furthermore, we recall the theorem of Crandall and Rabinowitz, see \cite{Crandall1971-wb}, in the formulation of \cite{Wahlen2006-qf} but stated for Banach manifolds.

\begin{theorem} \label{thm:CR}
Let $M$ be a smooth Banach manifold and $Y$ be a Banach space, $I \subset \R$ some open interval, and let $\mathcal{G}\colon I \times M \to Y$ be continuous. Let $w_0\in M$. Assume that:
\begin{enumerate}
    \item   $\mathcal{G}(\lambda,w_0)=0$ for all $\lambda\in I$.
    \item   The Fr\'echet derivatives $\mathrm{D}_\lambda \mathcal{G}$, $\mathrm{D}_w \mathcal{G}$, $\mathrm{D}^2_{\lambda w} \mathcal{G}$ exist and are continuous.
    \item   There exists $\lambda^*\in I$ and $w^*\in T_{w_0}M$ such that $N(\mathrm{D}_w \mathcal{G}(\lambda^*,w_0))=\Span(w^*)$ and $Y/R(\mathrm{D}_w \mathcal{G}(\lambda^*,w_0))$ is $1$-dimensional.
    \item   $\mathrm{D}^2_{\lambda w} \mathcal{G}(\lambda,w)|_{(\lambda,w) = (\lambda^*,w_0)}w^*\notin R(\mathrm{D}_w \mathcal{G}(\lambda^*,w)|_{w = w_0})$.
\end{enumerate}
Then there exists a continuous local bifurcation curve $\{(\lambda(s), w(s)):|s|<\varepsilon\}$ with $\varepsilon$ sufficiently small such that $(\lambda(0), w(0))=\left(\lambda^*, w_0\right)$ and
\begin{equation*}
    \{(\lambda, w) \in U : w \neq w_0,  \G(\lambda, w)=0\}=\{(\lambda(s), w(s)): 0<|s|<\varepsilon\}
\end{equation*}
for some neighbourhood $U$ of $\left(\lambda^*, w_0\right) \in I \times M$. Moreover, we have
\begin{equation*}
    w(s)=w_0 + s w^*+o(s) \quad \text { in } M, \; |s|<\varepsilon.
\end{equation*}
\end{theorem}

\begin{proof}[Proof of Theorem \ref{T2}]We introduce
\[
\Gamma = \left\{ \frac{2}{9\mu_k}: k\in\N\right\},
\]
with $(\mu_k)_{k\in\N}$ as in Proposition \ref{prop A}, and we write $\gamma_k = 2/(9\mu_k)$ for any integer $k$.

\medskip

\noindent

\textbf{(1)} We start by constructing  solutions for values $\gamma \not\in \Gamma$. We consider for any two fixed parameters $\delta^{\mathrm{in}},\delta^{\mathrm{out}} \in \R$ the mapping $\G_{\gamma} \colon \R \times \M^{\alpha+2} \to \H^{\alpha}_\sym(\Sp^2) /_\const$ defined by
    \begin{equation*}
       \G_{\gamma}(\eps,\eta) = \F\left( \gamma + \eps \delta^{\mathrm{in}} , \gamma+\eps \delta^{\mathrm{out}},\eta\right).
    \end{equation*}
   We recall that $\eta=0$ gives a solution of the jump equation \eqref{41} precisely if $\wein=\weout$, cf.~\eqref{trivial 0}. Therefore, we have $\G_{\gamma}(0,0)=\const$.
 Moreover, the mapping $\G_\gamma$ is continuously Fr\'echet differentiable in $\eta$ by Proposition \ref{P1} with derivative
    \begin{equation*}
        \D_\eta  \G_\gamma(0,\eta) |_{\eta = 0} = \frac{9}{2} \gamma \A\left( \frac{2}{9\gamma} \right) \colon T_0\M^{\alpha+2} \to \H^\alpha_{\mathrm{sym}}(\Sp^2) /_\const.
    \end{equation*}
which equals $\mathcal{A}^1$ for $\gamma=0$.

Since $\gamma \notin \Gamma$, by Proposition \ref{prop A} and Lemma \ref{lem:deriinv}, this operator is an isomorphism, and thus,   the implicit function theorem yields the existence of a constant $\epsilon_0 >0$ and of a unique continuously differentiable mapping $(-\epsilon_0,\epsilon_0) \to \M^{\alpha+2}$, $\epsilon \mapsto \eta_\epsilon$ such that $\G_\gamma(\eps,\eta_\eps) = \const$ is satisfied. 

In particular, if we choose $\din=0$ and $\dout=1$, the construction yields a small open interval $I_{\gamma}$ around $\gamma$, such that for any $\We\in I_{\gamma}$, there exists a unique $\eta = \eta(\gamma,\We)\in \M^{\alpha+2}$ so that the jump equation \eqref{41} is satisfied, $\F(\gamma,\We,\eta)=\const$. The uniqueness entails that for $\gamma=\We$, we recover the sperical vortex, $\eta(\gamma,\gamma)=0$. On the other hand, we recall that the trivial function $\eta=0$ solves \eqref{41}  precisely if condition \eqref{19} is satisfied.
Hence, if $\We\in I_{\gamma}\setminus\{\gamma\}$, our construction yields indeed non-spherical solutions, $\eta(\gamma,\We)\not=0$.

Let us now derive the asymptotic expansion of $\eta_\eps$ in the particular case where $\gamma=0$ and $\din,\dout\ge0$. As we are working modulo constants we have, as computed in \eqref{eq:curlpsiin0s1}, \eqref{eq:curlpsiout0s1} and \eqref{45},
\begin{equation} \label{eq:G2eps0}
    \G_0(\eps,0) = {\eps\frac{9}8} \left( \delta^{\text{in}} -  \delta^{\text{out}}\right)\sin^2\theta =    {-\eps \frac32\sqrt{\frac{\pi}5}(\din-\dout) Y_2^0(\theta)} 
\end{equation}
in $\H^\alpha(\Sp^2)/_\const$. Taking the derivative of the jump equation $\G_0(\eps,\eta_\eps) = 0$ (considered in the quotient space) and using Lemma \ref{lem:Dcurv} to note
\[
\mathrm{D}_{\eta} \G_\gamma(0,\eta)|_{\eta=0} = \mathrm{D}_{\eta}\cC(\eta)|_{\eta=0} = -(\laplace_{\Sp^2}+2\Id),
\]
which is invertible by  Lemma \ref{lem:deriinv}, 
we thus obtain that
\begin{align*}
    \D_{\eps}\eta_{\eps}|_{\eps=0}& = -\la (\D_{\eta}\cC(\eta)|_{\eta=0})^{-1},\D_{\eps}\G_0(\eps,0)|_{\eps=0}\ra \\
    &= -\frac32\sqrt{\frac{\pi}5}(\din-\dout) (\Delta_{\Sp^2} +2\Id)^{-1}Y_2^0\\
    & = \frac38\sqrt{\frac{\pi}5}(\din-\dout)  Y_2^0.
\end{align*}
Note that $\eta_\eps$ preserves the volume constraint at first order in $\eps$, and thus, in view of Lemma \ref{lem:M}, this derivative is an admissible tangent vector. Invoking \eqref{45} once more, we rewrite the latter as 
\[
\D_{\eps}\eta_{\eps}|_{\eps = 0} = \frac{3}{32}(\din-\dout)\left(3\cos^2\theta-1\right),
\]
which yields the asymptotic formula in the theorem.

We obtain the smoothness of $\eta$, because by construction and uniqueness, we have $\eta\in \H^{\beta}(\Sp^2)$ for any $\beta\ge 0$. The regularity can be upgraded to analyticity by using the general theory for elliptic free boundary problems, e.g., in \cite[Thm.\ 3.1]{KochLeoni}.

\medskip

\noindent
\textbf{(2)} We turn to the case where $\gamma=\gamma_k$. Our goal is to construct bifurcations from the spherical vortex solutions $\F(\gamma,\gamma,0) = \const$, cf.~\eqref{trivial 0}.
This time we consider the mapping 
    \begin{equation*}
        \cH(\gamma,\eta) = \F\left( \gamma,\gamma,\eta\right). 
    \end{equation*}
In  Proposition \ref{P1}, we have proven  that $\cH \colon \R \times \M^{\alpha+2} \to \H^\alpha_{\mathrm{sym}}(\Sp^2)/_\const$ is differentiable with derivative
    \begin{equation} \label{eq:DetaG}
        \D_\eta \cH(\gamma,\eta)|_{\eta = 0} = \frac{9}{2} \gamma \A\left( \frac{2}{9\gamma} \right) \colon T_0\M^{\alpha+2} \to \H^\alpha_{\mathrm{sym}}(\Sp^2) /_\const.
    \end{equation}
    We recall that we made the identification $\mu = \frac{2}{9\gamma} $. We claim that at each value $\gamma=\wein_k$, the theorem of Crandall and Rabinowitz is applicable. Indeed, we verify the assumptions of the theorem:
    \begin{enumerate}
        \item   We have $\cH(\gamma,0) = \const$ for all $\gamma \in \R$, cf.~\eqref{trivial 0}.
        \item   We have proven the continuous Fr\'echet differentiability with respect to $\eta$ in Proposition \ref{P1}. The derivatives with respect to $\gamma$ exist because $\mathcal{H}$ is affine in $\gamma$.
        \item   These properties are the content of Proposition \ref{prop A}.
        \item   The transversality property is precisely part e) of Proposition \ref{prop A}.
    \end{enumerate}
    Therefore, an application of Theorem \ref{thm:CR} yields the existence of a unique bifurcation  curve $s\mapsto (\gamma(s),\eta(s))\in I_k\times \M^{\alpha+2}$ with $(\gamma(0),\eta(0)) = (\gamma_k,0)$ such that $\cH(\gamma(s),\eta(s)) = \const$ and $\eta(s)\not=0$ for $s\not=0$. Here, $I_k$ is a small interval around $\gamma_k$.
    
    The regularity of the shape function $\eta$ is obtained by the same argument as in (1).   

\medskip
\noindent
We finally notice that by Proposition \ref{prop A} we obtain 
\begin{equation*}
    \gamma_1 \ge \frac{60060}{16510+ 2574 \sqrt{10}+945 \sqrt{65}} \approx 1.861. \qedhere
\end{equation*}
\end{proof}

\begin{remark} \label{rem:asymptoticG2}
     The identity in \eqref{eq:G2eps0} holds true for any $\gamma>0$. However, we are not able to solve $\langle \mathrm{D}_\eta\cH(\gamma,\eta)|_{\eta = 0}, v \rangle = \frac{9}{2}\gamma \A\left(\frac{2}{9\gamma} \right)v = Y_2^0$ explicitly for $v$. For this reason, we cannot provide an explicit asymptotic expansion in that situation.
\end{remark}

\bibliography{steady_bubbles.bib}
\bibliographystyle{abbrv}
\vfill
\end{document}